\documentclass{amsart}

\usepackage[pagewise]{lineno}

\usepackage{amsmath,amssymb,amsthm,fullpage,mathptmx, hyperref,dsfont,framed, graphicx, subcaption, textcomp}

\usepackage{tikz}
\usetikzlibrary{shapes.geometric}
\usetikzlibrary{calc}
\usetikzlibrary{scopes}
\usetikzlibrary{decorations.markings}

\tikzset{
every picture/.style={line width=0.8pt, >=stealth,
                       baseline=-3pt,label distance=-3pt},
dotnode/.style={fill=black,circle,minimum size=2.5pt, inner sep=1pt, outer
sep=0},
morphism/.style={circle,draw,thin, inner sep=1pt, minimum size=15pt,
                 scale=0.8},
small_morphism/.style={circle,draw,thin,inner sep=1pt,
                       minimum size=10pt, scale=0.8},
coupon/.style={draw,thin, inner sep=1pt, minimum size=18pt,scale=0.8},
regular/.style={densely dashed},
edge/.style={thick, dashed, draw=blue, text=black},
boundary/.style={thick,  draw=blue, text=black},
overline/.style={preaction={draw,line width=2mm,white,-}},
drinfeld center/.style={>=stealth,green!60!black, double
distance=1pt,text=black},
cell/.style={fill=black!10},
subgraph/.style={fill=black!30},
midarrow/.style={postaction={decorate},
                 decoration={
                    markings,
                    mark=at position #1 with {\arrow{>}},
                 }},
midarrow/.default=0.5
}

\newtheorem{thm}{Theorem}[section]
\newtheorem{lem}[thm]{Lemma}
\newtheorem{prop}[thm]{Proposition}

\newtheorem{defn}[thm]{Definition}
\newtheorem{rmk}[thm]{Remark}

\newcommand{\Hs}{H}
\newcommand{\lstar}{{^*}}

\DeclareMathOperator{\id}{id}
\DeclareMathOperator{\MCG}{MCG}
\DeclareMathOperator{\Mod}{Mod}
\DeclareMathOperator{\Vect}{Vec}

\DeclareMathOperator{\Hom}{Hom}
\DeclareMathOperator{\Obj}{Obj}

\DeclareMathOperator{\Img}{Im}
\DeclareMathOperator{\coev}{coev}
\DeclareMathOperator{\ev}{ev}
\DeclareMathOperator{\Gr}{Graph}
\DeclareMathOperator{\VGr}{VGraph}
\newcommand{\VV}{\mathbf{V}}       
\newcommand{\vgo}{\Vect_G^\omega}
\newcommand{\one}{1}

\newcommand{\Ga}{\Gamma}

\newcommand{\ph}{\varphi}

\newcommand{\Si}{\Sigma}

\begin{document}

\markboth{Paul Gustafson}
{Finiteness of Mapping Class Group Representations from Twisted Dijkgraaf-Witten Theory}


\title{Finiteness of Mapping Class Group Representations from Twisted Dijkgraaf-Witten Theory}

\author{Paul P. Gustafson}
\email{pgustafs@math.tamu.edu}
\address{Department of Mathematics,
    Texas A\&M University,
    College Station, TX
    U.S.A.}

\maketitle

\begin{abstract}
We show that any twisted Dijkgraaf-Witten representation of a mapping class group of an orientable, compact surface with boundary has finite image. This generalizes work of Etingof, Rowell and Witherspoon showing that the braid group images are finite \cite{erw}.  In particular, our result answers their question regarding finiteness of images of arbitrary mapping class group representations in the affirmative.

Our approach is to translate the problem into manipulation of colored graphs embedded in the given surface. To do this translation, we use the fact that any twisted Dijkgraaf-Witten representation associated to a finite group $G$ and 3-cocycle $\omega$ is isomorphic to a Turaev-Viro-Barrett-Westbury (TVBW) representation associated to the spherical fusion category $\Vect_G^\omega$ of twisted $G$-graded vector spaces. As shown by Kirillov, the representation space for this TVBW representation is canonically isomorphic to a vector space spanned by $\Vect_G^\omega$-colored graphs embedded in the surface \cite{kirillovStringNets}.   By analyzing the action of the Birman generators \cite{birman} on a finite spanning set of colored graphs, we find that the mapping class group acts by permutations on a slightly larger finite spanning set.  This implies that the representation has finite image.
\end{abstract}



\section{Introduction}
Given a spherical fusion category $\mathcal A$ over a field $k$ and an oriented compact surface $\Si$, possibly with boundary, the Turaev-Viro-Barrett-Westbury (TVBW) construction gives a projective representation of the mapping class group $\MCG(\Si)$ \cite{hep-th/9311155, TURAEV1992865}.  A natural problem is to determine the images of such representations.  In particular, we would like to know when such a representation has finite image.

It is conjectured that any TVBW mapping class group representation associated to a spherical fusion category $\mathcal A$ has finite image if and only if  $\mathcal A$ is weakly integral.  This conjecture is a modification of the Property F conjecture \cite{erw, nr}, which states that braid group representations coming from a braided monoidal category $\mathcal C$ should have finite image if and only if $\mathcal C$ is weakly integral. Instead of only considering braid group representations, one can consider mapping class groups of arbitrary orientable surfaces.  In this case, the input categories to construct the representations must be more specialized than just braided monoidal.  One can either apply the Reshitikhin-Turaev construction to a modular tensor category, or apply the TVBW construction to a spherical fusion category.  The former is more general than the latter since the Reshitikhin-Turaev construction for the Drinfeld center $Z(\mathcal A)$ of a spherical fusion category $\mathcal A$ yields the same representation as the TVBW construction for $\mathcal A$.  However, for the case considered in this paper, the simpler TVBW construction suffices.

In this paper, our input category is  $\mathcal A = \Vect_G^\omega$, the spherical fusion category of $G$-graded vector spaces with associativity modified by a cocycle $\omega \in Z^3(G, k^\times)$.  In this case, the TVBW construction corresponds to the twisted Dijkgraaf-Witten theory of \cite{dijkgraaf1990}.  The category $\Vect_G^\omega$ is integral, so the one expects its associated mapping class group representations to have finite image.  The main contribution of this paper is to verify this for arbitrary $G$ and $\omega$.

\textbf{Acknowledgments.}  This paper would not have been written without the guidance of my advisor, Eric Rowell.  I am also grateful to Zhenghan Wang and my father Robert Gustafson for their advice.

\section{Related Work}

The closest related work is a result of Etingof, Rowell, and Witherspoon who showed purely algebraically that the braid group representations associated to the modular category $\Mod(D^\omega(G))$ have finite images \cite{erw}.   The braid group $B_n$ is the mapping class group of a disk with $n$ marked points relative to its boundary, so they asked whether their result generalizes to arbitrary mapping class group representations associated to $\Mod(D^\omega(G))$. This paper answers their question affirmatively, using a different, more geometric approach.

Prior to the current work, certain specific cases had already been solved. In the case of the torus, Ng and Schauenburg's Congruence Subgroup Theorem implies the much stronger result that any Reshitikhin-Turaev representation of the mapping class group of the torus has finite image \cite{0806.2493}.   Another related result is due to Fjelstad and Fuchs \cite{fjfu}.  They showed that, given a surface with at most one boundary component, the mapping class group representations corresponding to the untwisted (i.e. $\omega = 1$) Dijkgraaf-Witten theory have finite image.  Their paper uses an algebraic method of Lyubashenko \cite{Lyubashenko1996} that gives a projective mapping class group representation to any factorizable ribbon Hopf algebra, in their case, the double $D(G)$. In our case, we instead consider the mapping class group action on a vector space of $\Vect_G^\omega$-colored embedded graphs defined by Kirillov \cite{kirillovStringNets}, yielding a simpler, geometric proof of the more general twisted case.

Bantay also calculated the images of certain representations of mapping class groups on the Hilbert space of an orbifold model associated to $D^\omega(G)$ \cite{bantay}.  These representations appear to coincide with the twisted Dijkgraaf-Witten representations. However, due to lack of proof, the precise connection is unclear.

\section{Background}
\subsection{The spherical fusion category $\Vect_G^\omega$}
The following definitions are well-known and can be found in, e.g., \cite{etingofTensor}.  Let $k$ be an  algebraically closed field of characteristic 0,  $G$ a finite group, and $\omega \in Z^3(G, k^\times)$ a 3-cocycle.    The spherical fusion category of $G$-graded $k$-vector spaces with associativity defined by $\omega$ is denoted $\Vect_G^\omega$.  The objects of this category are vector spaces with a decomposition $V = \bigoplus_{g \in G} V_g$. Morphisms are linear maps preserving the grading. The tensor product is defined by
$$ (V \otimes W)_g = \bigoplus_{x,y \in G, xy = g} V_x \otimes W_y. $$

For each $g \in G$, pick a 1-dimensional vector space $\delta_g \in \Obj(\Vect_G^\omega)$ concentrated in degree $g$.  The set $\{\delta_g : g \in G\}$ is a complete set of pairwise non-isomorphic representatives for the isomorphism classes of simple objects of $\Vect_G^\omega$.  We will sometimes abuse notation by referring to an object $\delta_g$ by the group element $g$.  We have $\one \cong \delta_1$, and  $\delta_g^* := \delta_{g^{-1}}$ with the coevaluation and evaluation maps defined below.

For the structural morphisms,  we follow \cite{math/0601012}.  We will treat the canonical isomorphisms $\delta_g \otimes \delta_h \cong \delta_{gh}$ as identities.    The associator $\alpha_{g,h,k}:(\delta_g \otimes \delta_h) \otimes \delta_k \to \delta_g \otimes (\delta_h \otimes \delta_k)$ is defined by
$$\alpha_{g,h,k} = \omega(g,h,k) \id_{ghk}.$$ 
The evaluator $\ev_g:\delta_g^* \otimes \delta_g \to 1$ is 
$$\ev_{g} = \omega(g^{-1},g,g^{-1}) \id_1.$$  
The coevaluator $\coev_g:1 \to \delta_g \otimes \delta_g^*$ is 
$$\coev_{g} = \id_1.$$ 
The pivotal structure $j_g:\delta_g \to \delta_g^{**}$ is 
$$j_{g} = \omega(g^{-1},g,g^{-1}) \id_{g}.$$

If $\omega$ and $\omega'$ are cohomologous cocycles, then $\Vect_G^\omega$ is monoidally equivalent to $\Vect_G^{\omega'}$ \cite{etingofTensor}.  This equivalence respects the pivotal structure, so extends to an equivalence of spherical categories.  It is a basic result in group cohomology that any cocycle $\omega \in Z^3(G, k^\times)$ is cohomologous to a cocycle taking values in $\mu_{|G|}$, the roots of unity of order $|G|$.  Thus, by replacing $\Vect^\omega_G$ with an equivalent spherical category, we assume  without loss of generality that $\Img(\omega) \subset \mu_{|G|}$ (as in \cite{erw}).  

\newcommand{\ee}{\mathbf{e}}       
\newcommand{\A}{\mathcal{A}}      
\newcommand{\st}{\; | \;}                               
\newcommand{\ttt}{\otimes}                              
\newcommand{\cc}[1]{\underset{\scriptstyle #1}{\circ}}
\newcommand{\ccc}[1]{\underset{\scriptstyle #1}{\bullet}}
\newcommand{\ti}{\tilde}
\newcommand{\ov}{\overline}
\newcommand{\del}{\partial}
\newcommand{\<}{\langle}
\renewcommand{\>}{\rangle}
\newcommand{\surjto}{\twoheadrightarrow}      
\newcommand{\injto}{\hookrightarrow}          
\newcommand{\isoto}{\xrightarrow{\sim}}       
\newcommand{\xxto}{\xrightarrow}              
\newcommand{\firef}[1]{Figure~{\rm\ref{#1}}}
\newcommand{\R}{\mathbb{R}}       

\subsection{Colored graphs}\label{s:colored}

The following definitions and theorem are due to Kirillov
\cite{kirillovStringNets}.  An unstrict version is recorded here for convenience (see \cite{ns} for definitions of the strictification monoidal functors and their quasi-inverses).  For any spherical fusion category $\mathcal A$ and surface $\Si$, Kirillov gives the following presentation of the Levin-Wen model as a vector space of
colored graphs modulo local relations.  He also proves that this space
is canonically isomorphic to the TVBW vector space associated to
$\Si$.  It is straightforward to check that this isomorphism, which
amounts to replacing a triangulation with its dual graph, commutes
with the mapping class group action.

We use the convention that a tensor product of multiple objects with parentheses
omitted correspond to the left-associative parenthesization.

We define the functor $\A^{\boxtimes n}\to \Vect$ by
\begin{equation}\label{e:vev}
\<V_1,\dots,V_n\>=\Hom_\A(\one,
V_1\otimes\dots\otimes V_n)
\end{equation}
for any collection $V_1,\dots, V_n$ of objects of $\A$. Note that pivotal
structure gives functorial isomorphisms
\begin{equation}\label{e:cyclic}
z\colon\<V_1,\dots,V_n\>\cong \<V_n, V_1,\dots,V_{n-1}\>
\end{equation}
where, up to associators and unitors,
$$z(\phi)= (j_{\lstar \lstar V_n} \otimes \id_{V_1 \otimes \cdots \otimes V_{n-1}} \otimes \ev_{V_n} ) \circ (\id_{\lstar \lstar V_n} \otimes \phi \otimes \id_{\lstar V_n} ) \circ  \coev_{\lstar \lstar V_n}  $$ 
and $z^n=\id$ (see \cite[Section 5.3]{BK}). Thus, up to a canonical
isomorphism, the space $\<V_1,\dots,V_n\>$ only depends on the cyclic order
of $V_1,\dots, V_n$.

We have a natural composition map 
\begin{equation}\label{e:composition}
\begin{aligned}
 \<V_1,\dots,V_n, X\> \otimes \<X^*, W_1,\dots,
W_m\>&\to\<V_1,\dots,V_n, W_1,\dots, W_m\> \\
\ph\otimes\psi & \mapsto \ph\cc{X}\psi= (\id_{V_1 \otimes \cdots \otimes V_{n}} \otimes (\ev_{X^*} \circ (j_X \otimes \id_{X^*}  )\otimes \id_{W_1 \otimes \cdots \otimes W_{m}} )\circ (\ph\otimes\psi),
\end{aligned}
\end{equation}
up to associators and unitors.  

We will consider finite directed graphs embedded in an oriented surface $\Si$
(which may have boundary); for such a
graph $\Ga$, let $E(\Ga)$ be the set of edges.

If $\Si$ has a boundary, the graph is allowed to have uncolored 
one-valent vertices on $\del \Si$ with exactly one incoming edge but no other common points with $\del \Si$; all other  vertices will  be called interior.  We will  call the edges of $\Ga$ terminating in these  one-valent vertices {\em legs}.   
\begin{defn}\label{d:coloring} Let $\Si$ be an oriented surface
(possibly with boundary) and $\Ga\subset \Si$ --- an embedded graph as
defined above.  A {\em coloring} of $\Ga$ is the
following data:

  \begin{itemize}
    \item Choice of an object $V(\ee)\in \Obj \A$ for every oriented edge $\ee\in E(\Ga)$ 
    \item Choice of a vector $\ph(v)\in \<V(\ee_1)^{\epsilon_1},\dots,V(\ee_n)^{\epsilon_n}\>$ 
      (see \eqref{e:vev})  for    every interior vertex $v$, where 
      $\ee_1, \dots, \ee_n$ are edges incident to $v$ taken in counterclockwise 
      order. The object $V(\ee_i)^{\epsilon_i} = V(\ee_i)$ if $\ee_i$ is outgoing from $v$ and $V(\ee_i)^*$ if $\ee_i$ is incoming. 
\end{itemize}

  An {\em isomorphism} $f$ of two colorings $\{V(\ee), \ph(v)\}$, $\{V'(\ee), \ph'(v)\}$ is a collection of isomorphisms $f_\ee \colon V(\ee)\cong V'(\ee)$  for which  $\ph'(v)=f\circ\ph(v)$, where the action of $f$ on $V(\ee)^*$ is by $(f_\ee^{-1})^*$.

We will denote the set of all colored graphs on a surface $\Si$ by
$\Gr(\Si)$.
\end{defn}


Note that if $\Si$ has a boundary, then every colored graph $\Ga$ defines
a collection of points $B=\{b_1,\dots, b_n\}\subset \del \Si$ (the
endpoints of the legs of $\Ga$)  and  a collection of objects $V_b\in \Obj\
\A$ for every $b \in B$: the colors of the legs of $\Ga$. We will the pair $(B, \{V_b\})$ by
$\VV=\Ga\cap \del\Si$ and call it a {\em boundary value}. We will denote  
$$
\Gr(\Si, \VV)=\text{set of all colored graphs in $\Si$ with boundary value
} \VV.
$$

We can also consider formal linear combinations of colored graphs. Namely,
for fixed boundary value $\VV$ as above, we will denote 
\begin{equation*}\label{e:vgr}
\VGr(\Si,\VV)=\{\text{formal linear combinations of graphs }\Ga\in
\Gr(\Si,\VV)\}
\end{equation*}
In particular, if $\del \Si=\varnothing$, then the only possible boundary
condition is trivial ($B=\varnothing$); in this case, we will just write
$\VGr(\Si)$.

The following theorem is a variation of result of Reshitikhin and Turaev. 
\begin{thm}\label{t:RT}
  There is a unique  way to assign to every colored
  planar graph $\Ga$ in a disk $D\subset \R^2$ a vector
  \begin{equation*}
    \<\Ga\>_D\in \langle V(\ee_1)^{\epsilon_i}, \dots, V(\ee_n)^{\epsilon_n} \rangle
  \end{equation*}
  where $\ee_1,\dots, \ee_n$ are the edges of $\Ga$ meeting the boundary
  of $D$ (legs), taken in counterclockwise order and $\epsilon_i$ are the corresponding orientations,
  so that that following conditions are satisfied:
  \begin{enumerate}
     \item $\<\Ga\>$ only depends on the isotopy  class of $\Ga$.

    \item If $\Ga$ is a single vertex colored by
          $\ph\in \langle V(\ee_1), \dots,  V(\ee_n)\rangle$, then $\<\Ga\>=\ph$.
     
    \item Local relations shown in \firef{f:local_rels1} hold.

\begin{figure}[ht]
$$
\begin{tikzpicture}
\node[morphism] (ph) at (0,0) {$\psi$};
\node[morphism] (psi) at (1,0) {$\ph$};
\node at (-0.7,0.1) {$\vdots$};
\node at (1.7,0.1) {$\vdots$};
\draw[->] (ph)-- +(220:1cm) node[pos=1.0,below,scale=0.8]
{$V_n$};
\draw[->] (ph)-- +(140:1cm) node[pos=1.0,above,scale=0.8]
{$V_1$};
\draw[->] (psi)-- +(40:1cm) node[pos=1.0,above,scale=0.8]
{$W_m$};
\draw[->] (psi)-- +(-40:1cm) node[pos=1.0,below,scale=0.8]
{$W_1$};
\draw[->] (ph) -- (psi) node[pos=0.5,above,scale=0.8] {$X$};
\end{tikzpicture}
=
\begin{tikzpicture}
\node[ellipse, thin, scale=0.8, inner sep=1pt, draw] (ph) at (0,0)
             {$\psi\cc{X}\ph$};
\node at (-0.8,0.1) {$\vdots$};
\node at (0.8,0.1) {$\vdots$};
\draw[->] (ph)-- +(220:1cm) node[pos=1.0,below,scale=0.8] {$V_n$};
\draw[->] (ph)-- +(140:1cm) node[pos=1.0,above,scale=0.8] {$V_1$};
\draw[->] (ph)-- +(40:1cm) node[pos=1.0,above,scale=0.8]  {$W_m$};
\draw[->] (ph)-- +(-40:1cm) node[pos=1.0,below,scale=0.8] {$W_1$};
\end{tikzpicture}
$$
\\
$$
\begin{tikzpicture}
\node[dotnode] (ph) at (0,0) {};
\node[dotnode] (psi) at (1.5,0) {};
\node at (-0.7,0.1) {$\vdots$};
\node at (2.2,0.1) {$\vdots$};
\draw[->] (ph)-- +(220:1cm) node[pos=1.0,below,scale=0.8] {$A_n$};
\draw[->] (ph)-- +(140:1cm) node[pos=1.0,above,scale=0.8] {$A_1$};
\draw[->] (psi)-- +(40:1cm) node[pos=1.0,above,scale=0.8] {$B_m$};
\draw[->] (psi)-- +(-40:1cm) node[pos=1.0,below,scale=0.8] {$B_1$};
\draw[out=45,in=135, midarrow] (ph) to (psi)
                node[above,xshift=-0.6cm, yshift=0.25cm, scale=0.8] {$V_k$};
\draw[ out=15,in=165, midarrow] (ph) to (psi);
\draw[ out=-15,in=195, midarrow] (ph) to (psi);
\draw[ out=-45,in=225, midarrow] (ph) to (psi) node[below, xshift=-0.6cm, yshift=-0.3cm, scale=0.8] {$V_1$};
\end{tikzpicture}
=
\begin{tikzpicture}
\node[dotnode] (ph) at (0,0) {};
\node[dotnode] (psi) at (1.5,0) {};
\node at (-0.7,0.1) {$\vdots$};
\node at (2.2,0.1) {$\vdots$};
\draw[->] (ph)-- +(220:1cm) node[pos=1.0,below,scale=0.8] {$A_n$};
\draw[->] (ph)-- +(140:1cm) node[pos=1.0,above,scale=0.8] {$A_1$};
\draw[->] (psi)-- +(40:1cm) node[pos=1.0,above,scale=0.8] {$B_m$};
\draw[->] (psi)-- +(-40:1cm) node[pos=1.0,below,scale=0.8] {$B_1$};
\draw[ ->] (ph) to (psi)
            node[above,xshift=-0.8cm,scale=0.8] {$V_1\otimes \dots\otimes V_k$};
\end{tikzpicture}
\qquad k\ge 0
$$
\\
$$
\begin{tikzpicture}
\node[ellipse, scale=0.8, inner sep=1pt, draw,thin] (ph) at (0,0)
{$\coev_{\lstar V}$};
\draw[->] (ph)-- +(180:1cm) node[pos=1.0,above,scale=0.8] {$\lstar V$};
\draw[->] (ph)-- +(0:1cm) node[pos=1.0,above,scale=0.8] {$V$};
\end{tikzpicture}
=
\begin{tikzpicture}
\draw[->] (0,0)-- (2,0) node[pos=0.5,above,scale=0.8] {$V$};
\end{tikzpicture}
$$
\caption{Local relations for colored graphs with associators and unitors suppressed.        }\label{f:local_rels1}

\end{figure}
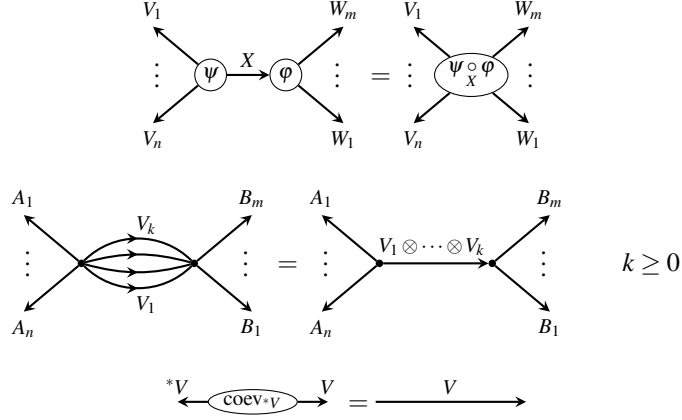

    Local relations should be understood as follows: for any pair 
    $\Ga, \Ga'$ of colored graphs which are identical  outside a subdisk 
	$D'\subset D$, and in this disk are homeomorphic to the graphs
    shown in  \firef{f:local_rels1},  we must have $\<\Ga\>=\<\Ga'\>$. 
   \end{enumerate}

    Moreover, so defined $\<\Ga\>$ satisfies the following properties:
    \begin{enumerate} 
    \item $\<\Ga\>$ is linear in color of each vertex $v$ \textup{(}for 
         fixed colors of edges and other vertices\textup{)}.
    \item $\<\Ga\>$ is additive in colors of edges as shown in 
          \firef{f:linearity}.

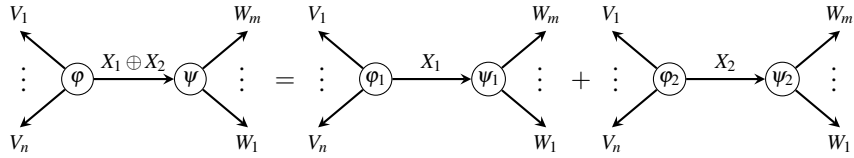
\begin{figure}[ht]
$$
\begin{tikzpicture}
\node[morphism] (ph) at (0,0) {$\ph$};
\node[morphism] (psi) at (1.5,0) {$\psi$};
\node at (-0.7,0.1) {$\vdots$};
\node at (2.2,0.1) {$\vdots$};
\draw[->] (ph)-- +(220:1cm) node[pos=1.0,below,scale=0.8] {$V_n$};
\draw[->] (ph)-- +(140:1cm) node[pos=1.0,above,scale=0.8] {$V_1$};
\draw[->] (psi)-- +(40:1cm) node[pos=1.0,above,scale=0.8] {$W_m$};
\draw[->] (psi)-- +(-40:1cm) node[pos=1.0,below,scale=0.8] {$W_1$};
\draw[->] (ph) -- (psi) node[pos=0.5,above,scale=0.8] {$X_1\oplus X_2$};
\end{tikzpicture}
=
\begin{tikzpicture}
\node[morphism] (ph) at (0,0) {$\ph_1$};
\node[morphism] (psi) at (1.5,0) {$\psi_1$};
\node at (-0.7,0.1) {$\vdots$};
\node at (2.2,0.1) {$\vdots$};
\draw[->] (ph)-- +(220:1cm) node[pos=1.0,below,scale=0.8]{$V_n$};
\draw[->] (ph)-- +(140:1cm) node[pos=1.0,above,scale=0.8]{$V_1$};
\draw[->] (psi)-- +(40:1cm) node[pos=1.0,above,scale=0.8]{$W_m$};
\draw[->] (psi)-- +(-40:1cm) node[pos=1.0,below,scale=0.8]{$W_1$};
\draw[->] (ph) -- (psi) node[pos=0.5,above,scale=0.8] {$X_1$};
\end{tikzpicture}
+
\begin{tikzpicture}
\node[morphism] (ph) at (0,0) {$\ph_2$};
\node[morphism] (psi) at (1.5,0) {$\psi_2$};
\node at (-0.7,0.1) {$\vdots$};
\node at (2.2,0.1) {$\vdots$};
\draw[->] (ph)-- +(220:1cm) node[pos=1.0,below,scale=0.8]{$V_n$};
\draw[->] (ph)-- +(140:1cm) node[pos=1.0,above,scale=0.8]{$V_1$};
\draw[->] (psi)-- +(40:1cm) node[pos=1.0,above,scale=0.8]{$W_m$};
\draw[->] (psi)-- +(-40:1cm) node[pos=1.0,below,scale=0.8]{$W_1$};
\draw[->] (ph) -- (psi) node[pos=0.5,above,scale=0.8] {$X_2$};
\end{tikzpicture}
$$
\caption{Linearity of $\<\Ga\>$. Here $\ph_1,\ph_2$ are compositions
of $\ph$ with projector $X_1\oplus X_2\to X_1$ (respectively, 
$X_1\oplus X_2\to X_2$), and similarly for $\psi_1,\psi_2$.
}\label{f:linearity}
\end{figure}
    \item If $\Ga,\Ga'$ are two isomorphic colorings of the same graph,
      then $\<\Ga\>=\<\Ga'\>$. 
    
    \item Composition property: if $D'\subset D$ is a subdisk such
      that $\del D'$ does not contain vertices of $\Ga$ and meets edges of
      $\Ga$ transversally, then   $\<\Ga\>_D$ will not change if we replace
      subgraph $\Ga\cap D'$ by a single vertex colored by
      $\<\Ga\cap D'\>_{D'}$.

  \end{enumerate}
The vector $\<\Ga\>$ is called the {\em evaluation} of $\Ga$.
\end{thm}

To define local relations between embedded graphs, Kirillov defines the space of null graphs as follows. Let
$\Ga=c_1\Ga_1+\dots+c_n\Ga_n$ be a formal linear
combination of colored graphs in $\Si$.  If there exists an embedded disk $D \subset M$ such that
\begin{enumerate}
  \item $\Ga$ is transversal to $\del D$ (i.e., no vertices of $\Ga_i$ 
      are on the boundary of $D$ and edges of each $\Ga_i$ meet 
      $\del D$ transversally),
  \item all $\Ga_i$ coincide outside of $D$,
  \item and $\<\Ga\>_D=\sum c_i\<\Ga_i\cap D\>_D=0$;
\end{enumerate}
then $\Ga$ is called a null graph.

\begin{defn}
The vector space $H := \Hs(\Si, \VV)$ associated to a oriented surface $\Si$ with boundary condition $\VV$ by the spherical fusion category $\mathcal A$ is the quotient space
 $$
   \Hs(\Si, \VV)=\VGr(\Si, \VV)/N(\Si, \VV)
  $$
  where $N(\Si, \VV)$ is  the subspace spanned by null graphs 
  (for all possible embedded disks  $D \subset \Si$). 
\end{defn}

\section{Results}

\newcommand{\B}{\mathcal B}



To show that the image of any $\Vect_G^\omega$ mapping class group representation is finite, we will analyze the action of the mapping class group on a finite collection of colored graphs that span the representation space $H$.  To define this spanning set, we will need the following definitions of simple morphisms and simple colored graphs.

\begin{defn} \label{def:simple_morphism}
Let $g_1, \ldots, g_n \in G$.  A morphism $\phi \in \langle g_1, \ldots, g_n \rangle$ will be called \emph{simple} if it is the composition of the isomorphism $\one \cong \delta_1$ and  tensor product isomorphisms of the form $\delta_{gh} \cong \delta_g \otimes \delta_h$.
\end{defn}

By MacLane's coherence theorem, there is a unique simple morphism in $\langle g_1, \ldots, g_n \rangle$ whenever $\prod_{i=1}^n g_i = 1$. This simple morphism is a canonical basis element for the 1-dimensional space $\langle g_1, \ldots, g_n \rangle$.  We will describe a map between such spaces as multiplication by a scalar, where the scalar is the matrix coefficient of the map with respect to these canonical bases.

\begin{defn} \label{def:simple_graph}
Let $\Gamma$ be a graph embedded in a surface $\Si$.  A $\Vect_G^\omega$ coloring $(V, \phi)$ of $\Gamma$ will be called \emph{simple} if the following conditions both hold: 
\begin{enumerate}
\item For every oriented edge $\ee \in E(\Gamma)$, there exists a group element $g(\ee) \in G$ such that
the coloring $V(e_i) = \delta_{g(\ee)}$.
\item If $v$ is an interior vertex of $\Gamma$, then there exists an enumeration  $\ee_1, \dots, \ee_n$ of the edges incident to $v$, taken in counterclockwise order, such that $\prod_{i=1}^n g(\ee_i)^{\epsilon_i} = 1$ and  the vertex label $\phi(v) \in \langle g(\ee_1)^{\epsilon_i}, \ldots, g(\ee_n)^{\epsilon_i} \rangle$ is a simple morphism, where $\epsilon_i = 1$ if $\ee_i$ is outgoing from $v$ and $-1$ if $\ee_i$ is incoming.
\end{enumerate}
\end{defn}

\begin{lem} \label{lem:z_simple}
Let $\phi \in \langle g_1, \ldots, g_n \rangle$ and $\psi \in \langle g_n, g_1, \ldots, g_{n-1} \rangle$ be simple morphisms.  Then $z(\phi) = \alpha \psi$, where $\alpha \in \mu_{|G|}$.
\end{lem}
\begin{proof}
The definition of the $z$-morphism in Equation \ref{e:cyclic} only involves tensors and compositions of structural morphisms -- i.e. associators, unitors, the pivotal $j$-morphism, evaluation, and coevaluation morphisms.  Since all the tensor factors in the codomain of $\phi$ are of the form $\delta_g$, the definition of $\Vect_G^\omega$ implies that each of the structural morphisms simply consist of multiplication by elements of the form $\omega(g,h,k)$ for some $g,h,k \in G$. Thus, $z(\phi) = \alpha \psi$ for some $\alpha$ which is a product of elements in $\Img(\omega)$.   Since $\Img(\omega) \subset \mu_{|G|}$, it follows that $\alpha \in \mu_{|G|}$.
\end{proof}

\begin{prop} \label{prop:omega}
Let $\Gamma$ be a simple colored graph embedded in a surface $\Sigma$.  Let $\Delta$ be the colored graph given by applying one of the three local moves in Figure \ref{f:local_rels1} to $\Gamma$.  Then
\begin{enumerate}
\item  each edge of $\Delta$ is labeled by $\delta_g$ for some $g \in G$, and
\item  there exists $\alpha \in \mu_{|G|}$ such that 
$$\Delta - \alpha \Delta' \in N(\Sigma, \VV),$$
where $\Delta'$ is a simple colored graph given by replacing each vertex label in $\Delta$ with a simple morphism.
\end{enumerate}
\end{prop}
\begin{proof}
The proof of (1) follows directly from the definition  of each local move. For (2), we'll consider each local move separately.  In each case, we need to show that $\Delta$ is equivalent to $\alpha \Delta'$ in $H$.  

For the first (edge contraction) local move in Figure \ref{f:local_rels1},  using the same notation as in the figure, the vertex label $\psi \cc{X} \phi$ in $\Delta$ is given by the following composition.  Since $\Gamma$ is simple, there exist integers $l,k$ and simple morphisms $\phi', \psi'$ such that $\psi = z^l(\psi')$ and $\phi = z^k(\phi')$. Then we repeatedly apply associators and the cyclic $z$-morphism of Equation \ref{e:cyclic} to $\phi$ and $\psi$ until the tensor factors of the codomain are rearranged in the order of the left hand side of Equation \ref{e:composition} and that $X$ and $X^*$ are isolated (not contained in any parentheses).  After applying the $\ev_X$ morphism, we reassociate until the new label $\ph\cc{X}\psi$ has the left-associated parenthesization.  Since every edge is labeled by an object of the form $\delta_g$, each structural morphism consists of multiplication by $\omega(g,h,k)$ for some $g,h,k \in G$.  Similarly, by Lemma \ref{lem:z_simple}, every $z$-morphism consists of multiplication by some $\beta \in \mu_{|G|}$.  Thus, the overall composition consists of multiplication by an element $\alpha \in \mu_{|G|}$.

For the second local move (tensoring parallel edges), there are two cases: $k = 0$ and $k > 0$.  In the $k = 0$ case,  we apply inverse unitors to each vertex label to introduce an edge labeled by the unit object, followed by reassocation.  In the $k > 0$ case, we need to reassociate to group together labels of the parallel edges then reassociate at the end.  For the same reasons as for the first move (every edge is labeled by a $\delta_g$), it follows that the result of this local move is also of the desired form.

For the third local move (adding a $\coev$-labeled vertex), the colored graph given by direct application of the local move to a simple graph is already simple, so we can pick $\alpha = 1$.  
\end{proof}

\subsection{No Boundary Case}

We first prove our theorem in the easier case where the surface $\Si$ is closed.  

\begin{thm}\label{thm:closed}
The image of any twisted Dijkgraaf-Witten representation of a mapping class group of an orientable, closed surface $\Si$ is finite.
\end{thm}

\begin{proof}
Let $\Gamma$ be a $\Vect_G^\omega$-colored graph embedded in $\Si$, and let $g \ge 1$ be the genus of $\Si$ (if $g = 0$, the mapping class group is trivial). Thinking of $\Si$ as a quotient of its fundamental $4g$-gon, by isotopy we may assume that the vertices of $\Gamma$ lie in the interior of the polygon, none of the edges of $\Gamma$ intersect the corners of the polygon, and that the edges of $\Gamma$ only meet the sides of the polygon transversally.  Applying the evaluation map of Theorem \ref{t:RT} on the interior of the polygon shows that $\Gamma$ is equivalent to a graph with a single vertex whose edges are simple closed curves, each of which intersect the boundary of the polygon precisely once. Note that we can always flip the orientation of an edge by adding a $\coev$-labeled vertex to the relevant edge and then contracting on the edge corresponding to the original edge label.   Thus, by using the local relations, we can replace all the edges intersecting a side with a single edge labeled by the tensor product of their labels.     If there are no edges intersecting a side, we can insert a single edge labeled by the group identity into $\Gamma$ that intersects only that side.  Thus, $\Gamma$ is equivalent to a colored graph with one vertex $v$ and edges $e_1, \ldots, e_{2g}$ corresponding to the standard generators of $\pi_1(M,v)$ as shown in Figure \ref{fig:span}.

By Theorem \ref{t:RT} and the definition of the quotient map identifying the sides of the fundamental polygon, the vertex $v$ is colored by an element $\phi(v) \in \Hom (\one, \bigotimes_{i=1}^g V(e_{2i-1}) \otimes V(e_{2i})  \otimes V(e_{2i-1})^* \otimes V(e_{2i})^*)$, where $V(e_i) \in \Obj(\Vect_G^\omega)$ is the coloring of the edge $e_i$.   

We claim that the representation space $H$ is spanned by the set of such colored graphs $\Gamma$ such that each $V(e_i)$ is simple. This follows from the additivity of the evaluation map of Theorem \ref{t:RT} in the direct sum.   Strictly speaking, we can only take advantage of the additivity on a disk, not on an edge $e_i$, which is a $v$-based loop. However, we can easily add a $\coev$-labeled vertex to any edge $e_i$, apply the additivity on one of the two resulting edges (which lies in an embedded disk), and then contract on the other edge to get the decomposition we want.

Since isomorphic colorings give the same evaluation, it follows that $H$ is spanned by colored graphs $\Gamma$ such that
each $V(e_i) = \delta_{g_i}$ for some $g_i \in G$.  For such $\Gamma$,  the space of possible $v$-colors $\Hom (\one, \bigotimes_{i=1}^g V(e_{2i-1}) \otimes V(e_{2i})  \otimes V(e_{2i-1})^* \otimes V(e_{2i})^*)$ is one-dimensional if $\prod_{i=1}^g [g_{2i-1}, g_{2i}] = 1$, and zero-dimensional otherwise.

By using the linearity with respect to the vertex label, we can further restrict to simple colored graphs $\Gamma$.  Thus, the representation space $H$ has a spanning set $S$ consisting of all simple colored graphs $\Gamma$ with one vertex $v$ and edges $e_1, \ldots, e_{2g}$ corresponding to the standard generators of $\pi_1(M,v)$.  Since there are $|G|$ objects of the form $\delta_g$ in $\Vect_G^\omega$ and at most $4g$ choices of simple morphisms labeling the vertex for a fixed edge labelling, the spanning set $S$ is finite.

\newdimen\R
\R=0.8cm

\begin{figure}
   \centering
    \begin{tikzpicture}[scale=3]    

      \draw (0:\R) \foreach \x in {45,90,...,359} {
                -- (\x:\R)
            } -- cycle;

    \begin{scope}[very thick,decoration={
    markings,
    mark=at position 0.5 with {\arrow{>}}}
    ]  
      \draw[postaction={decorate}]  (0, 0) --  (22: {0.923879*\R}) node[pos=.5,sloped,above]{$g$};
      \draw[postaction={decorate}]  (0, 0) --  (67: {0.923879*\R}) node[pos=.5,sloped,above]{$h$};
      \draw[postaction={decorate}]  (112: {0.923879*\R}) -- (0, 0) node[pos=.5,sloped,above]{$g$};
      \draw[postaction={decorate}]  (157: {0.923879*\R}) -- (0, 0) node[pos=.5,sloped,above]{$h$};
      \draw[postaction={decorate}]  (0, 0) --  (202: {0.923879*\R}) node[pos=.5,sloped,above]{$k$};
      \draw[postaction={decorate}]  (0, 0) --  (247: {0.923879*\R}) node[pos=.5,sloped,above]{$l$};
      \draw[postaction={decorate}]  (292: {0.923879*\R}) -- (0, 0) node[pos=.5,sloped,above]{$k$};
      \draw[postaction={decorate}]  (337: {0.923879*\R}) -- (0, 0) node[pos=.5,sloped,above]{$l$};
    \end{scope}
    \end{tikzpicture}
\caption{Element of the spanning set $S$ for a genus 2 surface}
\label{fig:span}
\end{figure}
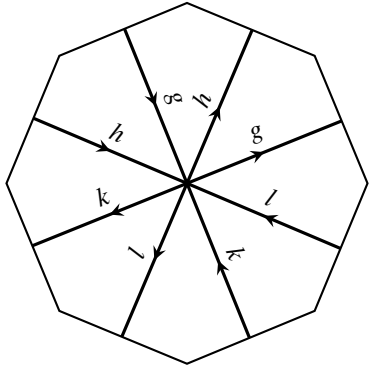

\begin{figure}
 \centering
 \includegraphics[width=.60\textwidth]{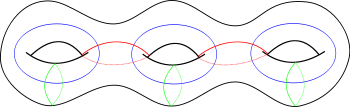}
 \caption{Simple closed curves for the Dehn twists in the Lickorish generating set for the mapping class group of a genus $3$ closed surface (Image source: \url{https://en.wikipedia.org/wiki/Dehn_twist})}
\label{fig:lickorish}
\end{figure}

The mapping class group of $\Si$ is generated by the Lickorish generating set consisting of Dehn twists around $3g-1$ simple closed curves.   These can be divided into two types of twists: the ones around a single hole (the blue and green curves in Figure \ref{fig:lickorish}), and the ones connecting two holes (the red curves). The action of a Dehn twist around a simple closed curve corresponds to cutting the manifold along the curve, holding one piece in place and twisting the other piece by $2\pi$ radians in a clockwise direction, then gluing the two pieces back together.

To understand the action of each type of Dehn twist on the representation space $H$, we will consider the action on the spanning set $S$.  First, we claim that we can apply local moves to any element of $S$ to get a colored graph of the form shown in the first subfigure of Figure \ref{fig:tikzTwist1_1}, where the unshown part of the fundamental polygon looks the same as in the definition of $S$.  Indeed, to pass from an arbitrary element of $S$, to a colored graph of the form shown in the first subfigure of $\ref{fig:tikzTwist1_1}$, we first add coevalution-labeled vertices to each  edge intersecting the three shown sides of the fundamental polygon.  Then connect the new vertices using edges labeled by the trivial object (this corresponds to applying the second local move in Figure \ref{f:local_rels1} with  $k = 0$), contract the connections to get one new vertex, and tensor together the edges connecting the old vertex to the new vertex.

The action of the first type of Dehn twist on an arbitrary element of $S$ is shown in the first two subfigures of Figure \ref{fig:tikzTwist1_1}.  
After applying the Dehn twist, we have the simple colored graph shown in the second subfigure of Figure \ref{fig:tikzTwist1_1}.  We then apply local moves in the remaining subfigures.  For example, to go from the second subfigure to the third, we first apply the third local move of Figure \ref{f:local_rels1} to add a $\coev$-labeled vertex to the top left $g$-labeled edge. Note that we have flipped the orientation of one of the edges incident to this new vertex, dualizing its object label.  We then apply the second local move (tensoring edges) with the number of parallel edges $k = 0$ to add an edge labeled by the trivial object between the new vertex and the old one. To go from the third to the fourth subfigure, we apply the edge contraction local move on the new edge.  Lastly, we get to the fifth subfigure by applying the tensoring edges local move again (strictly speaking, this is not a valid move since it does not take place on a disk, but one can easily add a $\coev$-labeled vertex to each of the two parallel edges, connect them, contract the connection, tensor together each of the two pairs of parallel edges, and contract one of the resulting edges to get the same result). By repeated application of Proposition $\ref{prop:omega}$, the resulting colored graph is equivalent to $\beta \Delta$, for some $\beta \in \mu_{|G|}$ and  $\Delta \in S$.    Thus, the first type of Dehn twist maps $S$ to $\mu_{|G|} S$.

An analogous proof works for the second type of Dehn twist shown in Figure \ref{fig:tikzTwist2_0}. Thus, the image of any such mapping class group representation is a quotient of the group of permutations of the finite set $\mu_{|G|} S$, hence finite.
\end{proof}

\begin{rmk} 
When $\omega = 1$ and $\Si$ is closed, this representation is a permutation representation.  
\end{rmk}

\begin{proof}
Under the assumption that the representation in \cite{fjfu} coincides with ours, this fact follows from Theorem 2.6 in \cite{fjfu}, but we can also prove it directly.  We first note that $G$ acts on $S$ by simultaneous conjugation of all edge labels by a single element $g \in G$.  If $s \in S$ and $g \in G$, then we can retrieve $s$ from $gs$ by separating two oppositely oriented, $g$-labeled edges from each edge in the embedded graph $gs$.  This results in a loop labeled by $g$, whose evaluation is $1$.  Thus $gs$ is equivalent to $s$.  Moreover, the cardinality $|S/G| = |\Hom(\pi_1(M), G)|/|G|$ is equal to the dimension of the untwisted Dijkgraaf-Witten representation space $H$ \cite{dijkgraaf1990}.  Hence, $S/G$ is a basis for $H$.  The mapping class group action on $S$ commutes with the $G$-action, so the mapping class group permutes $S/G$, i.e. $H$ is a permutation representation.
\end{proof}

\newcommand{\nc}{\newcommand}
\newcommand{\rnc}{\renewcommand}


     \nc{\lcx}{-0.5}
     \nc{\lcy}{0.866}
     \nc{\rcx}{-\lcx}
     \nc{\rcy}{\lcy}

     \nc{\makeBdy}{
       \begin{scope}[very thick,decoration={
             markings,
             mark=at position 0.5 with {\arrow{>}}}
         ]  

         \draw (-1,0) -- (\lcx, \lcy); 
         \draw (\lcx, \lcy) -- (\rcx, \rcy); 
         \draw  (1, 0) -- (\rcx, \rcy);
       \end{scope}
     }

    \nc{\lcutx}{-0.6}
    \nc{\lcuty}{0.6928}
    \nc{\lcut}{(\lcutx, \lcuty)}
    \nc{\rcutx}{-\lcutx}
    \nc{\rcuty}{\lcuty}
    \nc{\rcut}{(\rcutx, \rcuty)}

    \nc{\mvx}{0}
    \nc{\mvy}{0.2}
    \nc{\mv}{(\mvx, \mvy)}

    \nc{\outEdge}{\draw[postaction={decorate}]  (0, 0) -- \mv node[pos=.5, right]{$hgh^{-1}$};}

    \nc{\mtopx}{0}
    \nc{\mtopy}{0.866}
    \nc{\mtop}{(\mtopx, \mtopy)}

\begin{figure}
\centering

\begin{tabular}{|c|c|}

\hline

    \begin{tikzpicture}[scale=3]    

      \makeBdy

    \begin{scope}[very thick,decoration={
    markings,
    mark=at position 0.5 with {\arrow{>}}}
    ]  

      \draw[loosely dashed] \lcut -- \rcut;


      \draw[postaction={decorate}]   \mv -- (0.75, 0.433) node[pos=.5,sloped,above]{$h$};
      \draw[postaction={decorate}]   \mv -- \mtop node[pos=.5,left]{$g$};
      \draw[postaction={decorate}]  (-0.75, 0.433) -- \mv node[pos=.5,sloped,above]{$h$};
      \outEdge

    \end{scope}
    \end{tikzpicture}

&

    \begin{tikzpicture}[scale=3]
    
      \makeBdy

    \begin{scope}[very thick,decoration={
    markings,
    mark=at position 0.5 with {\arrow{>}}}
    ]

        \draw[postaction={decorate}]   \mv -- (0.75, 0.433) node[pos=.5,sloped,above]{$h$};
        \draw[postaction={decorate}]  (-0.75, 0.433) -- \mv node[pos=.5,sloped,above]{$h$};
        \outEdge

        \draw[postaction={decorate}]  \mv -- \rcut
        node[pos=.5,sloped,above]{$g$};
        \draw[postaction={decorate}]  \lcut -- \mtop
        node[pos=.5,sloped,below]{$g$};

    \end{scope}
    \end{tikzpicture}

\\ \hline

      \begin{tikzpicture}[scale=3]
    
      \makeBdy

    \begin{scope}[very thick,decoration={
    markings,
    mark=at position 0.5 with {\arrow{>}}}
    ] 

        \draw[postaction={decorate}]   \mv -- (0.75, 0.433) node[pos=.5,sloped,above]{$h$};
        \draw[postaction={decorate}]  (-0.75, 0.433) -- \mv node[pos=.5,sloped,above]{$h$};

        \outEdge

        \draw[postaction={decorate}]  \mv -- \rcut
        node[pos=.5,sloped,above]{$g$};
        \draw[postaction={decorate}]  \lcut --  ({(\lcutx + \mtopx)/2} , {(\lcuty + \mtopy)/2})
        node[pos=.5,sloped,below]{$g$};
        \draw[postaction={decorate}]  ({(\lcutx + \mtopx)/2} , {(\lcuty + \mtopy)/2}) -- \mtop
        node[pos=.5,sloped,below]{$g$};

        \draw  \mv --  ({(\lcutx + \mtopx)/2} , {(\lcuty + \mtopy)/2});

    \end{scope}
    \end{tikzpicture}

&

    
    \begin{tikzpicture}[scale=3]

      \makeBdy

      \begin{scope}[very thick,decoration={
            markings,
            mark=at position 0.5 with {\arrow{>}}}
        ] 

        \draw[postaction={decorate}]   \mv -- (0.75, 0.433) node[pos=.5,sloped,above]{$h$};
        \draw[postaction={decorate}]  (-0.75, 0.433) -- \mv node[pos=.5,sloped,above]{$h$};

        \outEdge

        \draw[postaction={decorate}]  \mv -- \rcut
        node[pos=.5,sloped,above]{$g$};
        \draw[postaction={decorate}]  \lcut -- \mv
        node[pos=.5,sloped,above]{$g$};
        \draw[postaction={decorate}]  \mv -- \mtop
        node[pos=.5,right]{$g$};

    \end{scope}
    \end{tikzpicture}

\\ \hline


    \begin{tikzpicture}[scale=3]

     \makeBdy
    
    \begin{scope}[very thick,decoration={
    markings,
    mark=at position 0.5 with {\arrow{>}}}
    ]  
        \draw[postaction={decorate}]   \mv -- (0.75, 0.433) node[pos=.5,sloped,above]{$hg$};
        \draw[postaction={decorate}]   \mv -- \mtop node[pos=.5,left]{$g$};
        \draw[postaction={decorate}]  (-0.75, 0.433) -- \mv node[pos=.5,sloped,above]{$hg$};
        \outEdge

    \end{scope}
    \end{tikzpicture}

&

\\ \hline
\end{tabular}

\caption{Using local moves to calculate the action of the first type of Dehn twist on an arbitrary element of the spanning set $S$. Read from left to right, then top to bottom.  Unlabeled interior edges are colored by the group identity element.  The Dehn twist is performed along the dashed simple closed curve.  The first two subfigures show the action of the Dehn twist.  The last three show the local moves relating the image of the Dehn twist to another element of $S$.}
\label{fig:tikzTwist1_1}
\end{figure}
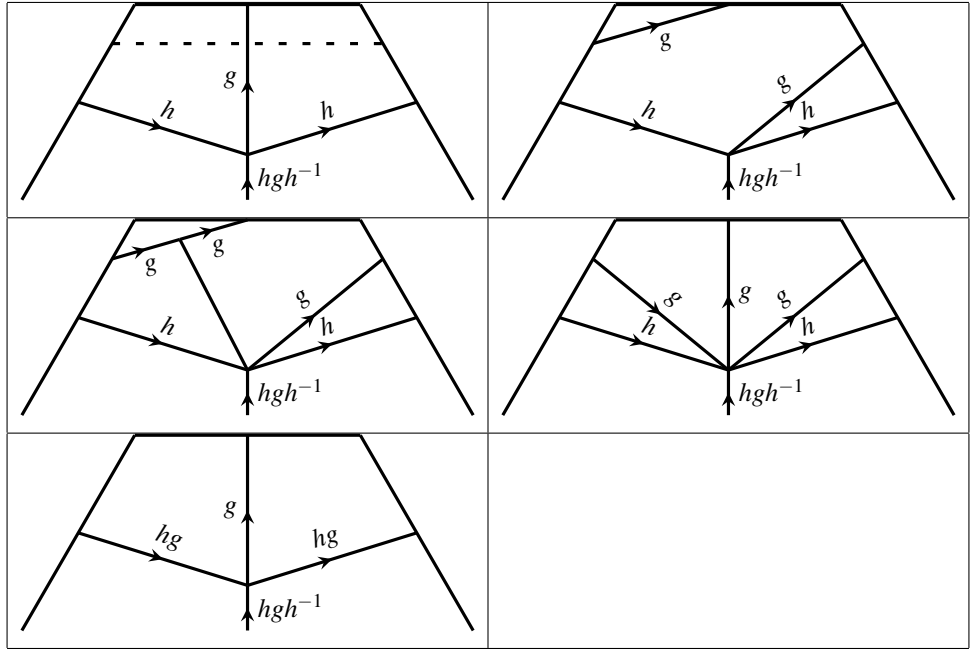


\nc{\makeBdyTwo}{
  \begin{scope}[very thick,decoration={
    markings,
    mark=at position 0.5 with {\arrow{>}}}
    ]  
    \path[draw]
    (1.0,          0.0) --       
    (0.92388,      0.382683)  -- 
    (0.707107,     0.707107) --  
    (0.382683,     0.92388) --   
    (0,            1.0)   --     
    (-0.382683,    0.92388) --   
    (-0.707107 ,   0.707107)  -- 
    ( -0.92388,    0.382683) --  
    ( -1.0     ,   0);           

  \end{scope}
}

\nc{\mvTwo}{ (-0.312076, 0.753418) }

\nc{\outEdgeTwo}{\draw[postaction={decorate}]  (0, 0) -- \mv node[pos=.3, right]{$[a,b][c,d]$};}

\nc{\cutOneX}{{0.8*(-0.92388) + 0.2*(-0.707107)}}
\nc{\cutOneY}{{0.8*(0.382683) + 0.2*(0.707107)}}
\nc{\cutOne}{(\cutOneX, \cutOneY)}

\nc{\cutTwoX}{{0.8*(0.382683) + 0.2*(0)}}
\nc{\cutTwoY}{{0.8*(0.92388) + 0.2*(1.0)}}
\nc{\cutTwo}{(\cutTwoX, \cutTwoY)}

\nc{\cutThreeX}{{0.8*(0.707107 ) + 0.2*(0.92388 )}}
\nc{\cutThreeY}{{0.8*(0.707107 ) + 0.2*(0.382638 )}}
\nc{\cutThree}{(\cutThreeX, \cutThreeY)}

\nc{\cutFourX}{{0.2*( 0.382683) + 0.8*(0.707107)}}
\nc{\cutFourY}{{0.2*( 0.92388) + 0.8*(0.707107)}}
\nc{\cutFour}{(\cutFourX, \cutFourY)}

\nc{\cutFiveX}{{0.2*( 1.0) + 0.8*(0.92388)}}
\nc{\cutFiveY}{{0.2*( 0.0) + 0.8*(0.382638)}}
\nc{\cutFive}{(\cutFiveX, \cutFiveY)}

\nc{\cutSixX}{{0.2*( 0.707107) + 0.8*(0.92388)}}
\nc{\cutSixY}{{0.2*( 0.707107) + 0.8*(0.382683)}}
\nc{\cutSix}{(\cutSixX, \cutSixY)}

\nc{\cutSevenX}{{0.2*( 0.382683) + 0.8*(0.0)}}
\nc{\cutSevenY}{{0.2*( 0.92388) + 0.8*(1.0)}}
\nc{\cutSeven}{(\cutSevenX, \cutSevenY)}

\nc{\cutEightX}{{0.2*( -0.382683) + 0.8*(0.0)}}
\nc{\cutEightY}{{0.2*( 0.92388) + 0.8*(1.0)}}
\nc{\cutEight}{(\cutEightX, \cutEightY)}

\begin{figure}
\centering

\begin{tabular}{|c|c|}

\hline
    \begin{tikzpicture}[scale=3]

     \makeBdyTwo

     \draw[dotted] \cutOne -- \cutTwo;
     \draw[dotted] \cutThree -- \cutFour;
     \draw[dotted] \cutFive -- \cutSix;
     \draw[dotted] \cutSeven -- \cutEight;
         
    \begin{scope}[very thick,decoration={
    markings,
    mark=at position 0.5 with {\arrow{>}}}
    ]  
       \draw[postaction={decorate}]  \mv -- (   0.96194 , 0.191342) node[pos=.5,sloped,above]{$a$};
       \draw[postaction={decorate}]  \mv -- ( 0.815493,  0.544895) node[pos=.5,sloped,above]{$b$};
       \draw[postaction={decorate}] (0.544895,  0.815493) -- \mv  node[pos=.5,sloped,above]{$a$};
       \draw[postaction={decorate}]   (0.191342,  0.96194) -- \mv  node[pos=.4,sloped,above]{$b$}; 
       \draw[postaction={decorate}]  \mv -- (-0.191342,  0.96194) node[pos=.4,sloped,above]{$c$}; 
       \draw[postaction={decorate}]  \mv -- (-0.544895,  0.815493) node[pos=.5,sloped,above]{$d$}; 
       \draw[postaction={decorate}] (-0.815493,  0.544895) -- \mv  node[pos=.5,sloped,above]{$c$};
       \draw[postaction={decorate}] (-0.96194,   0.191342) -- \mv  node[pos=.5,sloped,above]{$d$};

         \outEdgeTwo
    \end{scope}
    \end{tikzpicture}

&

    
    \begin{tikzpicture}[scale=3]

     \makeBdyTwo

     \draw[dotted] \cutOne -- \cutTwo;
     \draw[dotted] \cutThree -- \cutFour;
     \draw[dotted] \cutFive -- \cutSix;
     \draw[dotted] \cutSeven -- \cutEight;
         
    \begin{scope}[very thick,decoration={
    markings,
    mark=at position 0.5 with {\arrow{>}}}
    ]  
       \draw[postaction={decorate}]  \mv -- \mvTwo node[pos=.5,sloped,above]{\tiny $g := b^{-1}cdc^{-1}$};

       \draw[postaction={decorate}]  \mv -- (   0.96194 , 0.191342) node[pos=.5,sloped,above]{$a$};
       \draw[postaction={decorate}]  \mv -- ( 0.815493,  0.544895) node[pos=.5,sloped,above]{$b$};
       \draw[postaction={decorate}] (0.544895,  0.815493) -- \mv  node[pos=.5,sloped,above]{$a$};
       \draw[postaction={decorate}]   (0.191342,  0.96194) -- \mvTwo  node[pos=.5,sloped,above]{$b$}; 
       \draw[postaction={decorate}]  \mvTwo -- (-0.191342,  0.96194) node[pos=.5,sloped,above]{$c$}; 
       \draw[postaction={decorate}]  \mvTwo -- (-0.544895,  0.815493) node[pos=.5,sloped,above]{$d$}; 
       \draw[postaction={decorate}] (-0.815493,  0.544895) -- \mvTwo  node[pos=.5,sloped,above]{$c$};
       \draw[postaction={decorate}] (-0.96194,   0.191342) -- \mv  node[pos=.5,sloped,above]{$d$};

         \outEdgeTwo
    \end{scope}
    \end{tikzpicture}

\\ \hline

    
    \begin{tikzpicture}[scale=3]

     \makeBdyTwo
         
    \begin{scope}[very thick,decoration={
    markings,
    mark=at position 0.5 with {\arrow{>}}}
    ]  
       \draw[postaction={decorate}]  \mv -- \cutTwo node[pos=.5,sloped,above]{$g$};
       \draw[postaction={decorate}]  \cutOne -- \mvTwo node[pos=.5,sloped,below]{$g$};
       \draw[postaction={decorate}]  \cutThree -- \cutFour node[pos=.5,sloped,below]{$g$};
       \draw[postaction={decorate}]  \cutFive -- \cutSix node[pos=.5,sloped,below]{$g$};
       \draw[postaction={decorate}]  \cutSeven -- \cutEight node[pos=.5,sloped,red]{$g$};

       \draw[postaction={decorate}]  \mv -- (   0.96194 , 0.191342) node[pos=.5,sloped,above]{$a$};
       \draw[postaction={decorate}]  \mv -- ( 0.815493,  0.544895) node[pos=.5,sloped,above]{$b$};
       \draw[postaction={decorate}] (0.544895,  0.815493) -- \mv  node[pos=.5,sloped,above]{$a$};
       \draw[postaction={decorate}]   (0.191342,  0.96194) -- \mvTwo  node[pos=.5,sloped,below]{$b$}; 
       \draw[postaction={decorate}]  \mvTwo -- (-0.191342,  0.96194) node[pos=.5,sloped,above]{$c$}; 
       \draw[postaction={decorate}]  \mvTwo -- (-0.544895,  0.815493) node[pos=.5,sloped,above]{$d$}; 
       \draw[postaction={decorate}] (-0.815493,  0.544895) -- \mvTwo  node[pos=.5,sloped,above]{$c$};
       \draw[postaction={decorate}] (-0.96194,   0.191342) -- \mv  node[pos=.5,sloped,above]{$d$};

         \outEdgeTwo
    \end{scope}
    \end{tikzpicture}

&

    
    \begin{tikzpicture}[scale=3]

     \makeBdyTwo
         
    \begin{scope}[very thick,decoration={
    markings,
    mark=at position 0.5 with {\arrow{>}}}
    ]  
       \draw[postaction={decorate}]  \mv -- \cutTwo node[pos=.5,sloped,above]{$g$};
       \draw[postaction={decorate}]  \cutThree -- \mv node[pos=.3,sloped,above]{$g$};

       \draw[postaction={decorate}]  \mv -- ( 0.96194 , 0.191342) node[pos=.8,sloped,above]{$ag^{-1}$};
       \draw[postaction={decorate}]  \mv -- ( 0.815493,  0.544895) node[pos=.7,sloped,below]{$gb$};
       \draw[postaction={decorate}] (0.544895,  0.815493) -- \mv  node[pos=.2,sloped,above]{$ag^{-1}$};
       \draw[postaction={decorate}]   (0.191342,  0.96194) -- \mvTwo  node[pos=.5,sloped,below]{$gb$}; 
       \draw[postaction={decorate}]  \mvTwo -- (-0.191342,  0.96194) node[pos=.5,sloped,above]{$gc$}; 
       \draw[postaction={decorate}]  \mvTwo -- (-0.544895,  0.815493) node[pos=.5,sloped,above]{$d$}; 
       \draw[postaction={decorate}] (-0.815493,  0.544895) -- \mvTwo  node[pos=.5,sloped,below]{$gc$};
       \draw[postaction={decorate}] (-0.96194,   0.191342) -- \mv  node[pos=.5,sloped,above]{$d$};

         \outEdgeTwo
    \end{scope}
    \end{tikzpicture}

\\ \hline 


    \begin{tikzpicture}[scale=3]

     \makeBdyTwo
         
    \begin{scope}[very thick,decoration={
    markings,
    mark=at position 0.5 with {\arrow{>}}}
    ]  
       \draw[postaction={decorate}]  \mv -- \cutTwo node[pos=.8,sloped,red]{$g$};
       \draw[postaction={decorate}]  \cutThree -- \mv node[pos=.2,sloped,above]{$g$};

       \draw[postaction={decorate}]  \mv -- ( 0.96194 , 0.191342) node[pos=.8,sloped,above]{$ag^{-1}$};
       \draw[postaction={decorate}]  \mv -- ( 0.815493,  0.544895) node[pos=.7,sloped,below]{$gb$};
       \draw[postaction={decorate}] (0.544895,  0.815493) -- \mv  node[pos=.2,sloped,above]{$ag^{-1}$};
       \draw[postaction={decorate}]   (0.191342,  0.96194) -- \mv  node[pos=.2,sloped,above]{$gb$}; 
       \draw[postaction={decorate}]  \mv -- (-0.191342,  0.96194) node[pos=.8,sloped,below]{$gc$}; 
       \draw[postaction={decorate}]  \mv -- (-0.544895,  0.815493) node[pos=.7,sloped,above]{$d$}; 
       \draw[postaction={decorate}] (-0.815493,  0.544895) -- \mv  node[pos=.3,sloped,above]{$gc$};
       \draw[postaction={decorate}] (-0.96194,   0.191342) -- \mv  node[pos=.3,sloped,above]{$d$};

         \outEdgeTwo
    \end{scope}
    \end{tikzpicture}

&


    \begin{tikzpicture}[scale=3]

     \makeBdyTwo
         
    \begin{scope}[very thick,decoration={
    markings,
    mark=at position 0.5 with {\arrow{>}}}
    ]  

       \draw[postaction={decorate}]  \mv -- ( 0.96194 , 0.191342) node[pos=.8,sloped,above]{$ag^{-1}$};
       \draw[postaction={decorate}]  \mv -- ( 0.815493,  0.544895) node[pos=.8,sloped,above]{$gbg^{-1}$};
       \draw[postaction={decorate}] (0.544895,  0.815493) -- \mv  node[pos=.2,sloped,above]{$ag^{-1}$};
       \draw[postaction={decorate}]   (0.191342,  0.96194) -- \mv  node[pos=.2,sloped,above]{$gbg^{-1}$}; 
       \draw[postaction={decorate}]  \mv -- (-0.191342,  0.96194) node[pos=.8,sloped,below]{$gc$}; 
       \draw[postaction={decorate}]  \mv -- (-0.544895,  0.815493) node[pos=.7,sloped,above]{$d$}; 
       \draw[postaction={decorate}] (-0.815493,  0.544895) -- \mv  node[pos=.3,sloped,above]{$gc$};
       \draw[postaction={decorate}] (-0.96194,   0.191342) -- \mv  node[pos=.3,sloped,above]{$d$};

         \outEdgeTwo
    \end{scope}
    \end{tikzpicture}

\\ \hline
\end{tabular}

\caption{Using local moves to calculate the action of the second type of Dehn twist on an arbitrary element of the spanning set $S$. Read from left to right, then top to bottom.   The Dehn twist is performed along the dashed simple closed curve.  The first two subfigures show application of local moves prior to the Dehn twist action. The third shows the action of the twist.  The last three show the local moves relating the image of the Dehn twist to another element of $S$.}
\label{fig:tikzTwist2_0}
\end{figure}
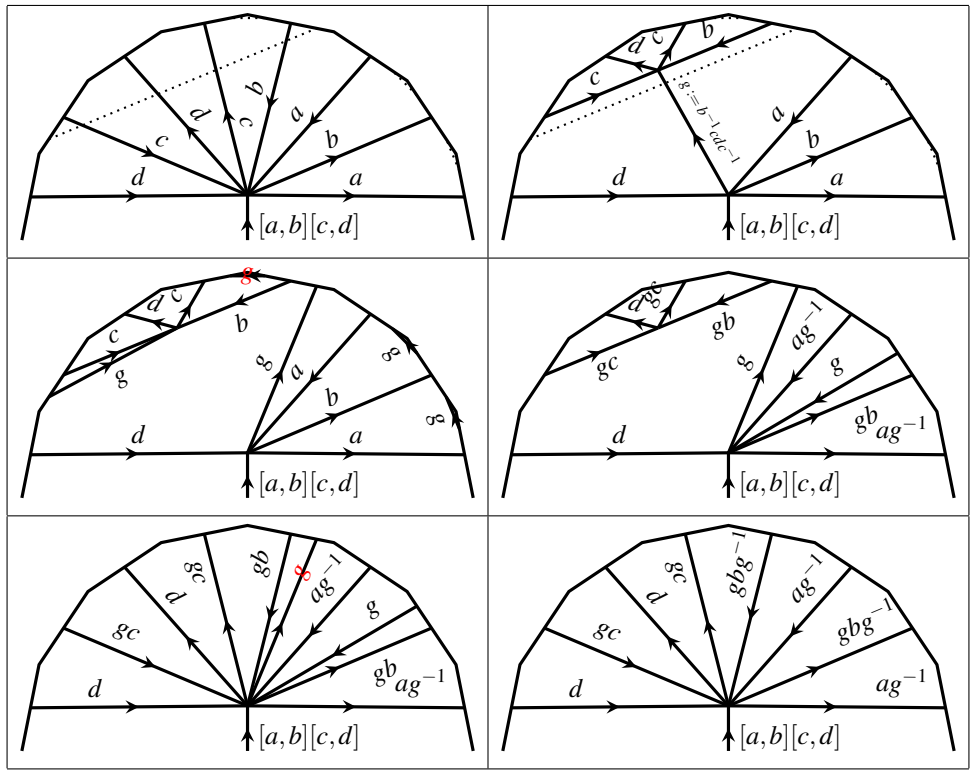

\subsection{Boundary Case}

When $\Si$ has boundary, we denote by $\MCG(\Si)$ the group of isotopy classes of homeomorphisms fixing the boundary of $\Si$ setwise. Given any labelling of the boundary  by objects in the Drinfeld center, $l : \pi_0(\partial M) \to \Obj(Z(\vgo))$, we get a mapping class group representation. Picking one point on each boundary component, we get a set of points $B \subset \partial M$ and bijection $b: B \to \pi_0(\partial M)$.  The representation space $H = \Hs(\Si, \VV)$ corresponding to the labelling $l$ has boundary condition $\VV = (B, F \circ l \circ b)$, where $F$ is the forgetful functor $F: Z(\vgo) \to \vgo$.  The local relations used above still hold in this representation space \cite{kirillovStringNets}.

By a similar argument as in the proof of the Theorem $\ref{thm:closed}$, any such representation space has a finite spanning set $S$ consisting of all simple colored graphs with a single vertex, loops for each of the usual generators of the fundamental group of $\Si$, and a leg from the vertex to each of the boundary components.  

Let $N$ denote the closed surface obtained by filling in all the boundary components of $\Si$ with disks. The mapping class group $\MCG(\Si)$ is generated by the same Dehn twists as $\MCG(N)$, as well as braids interchanging boundary components and mapping classes corresponding to dragging a boundary component along a representative of a standard generator of $\pi_1(N)$ \cite{birman}.  As in the proof of Theorem \ref{thm:closed}, applying any of these generators of $\MCG(\Si)$ to a colored graph in $S$ yields an element in $\mu_{|G|}S$ (see Figures \ref{fig:braid1} and \ref{fig:drag1}).  Since the braid group is also generated by such braids, we have the following theorem.

\nc{\rightX}{0.5}  
\nc{\leftX}{-\rightX} 
\nc{\pY}{0.5}  
\nc{\loopR}{0.3}
\renewcommand{\mv}{(0, -0.2)}

\nc{\makeBase}{
       \draw[very thick]  \mv -- ({\leftX - \loopR}, \pY);
       \draw[very thick]   ({\leftX + \loopR}, \pY) -- \mv;
       
       \begin{scope}[very thick,decoration={
             markings,
             mark=at position 0.5 with {\arrow{>}}}
         ]  

         \draw[postaction={decorate}]  \mv -- (\leftX, \pY) node[pos=.5,above]{$l$};

       \end{scope}

       \draw[very thick]  \mv -- ({\leftX - 4*\loopR}, \pY);
       \draw[very thick]   ({\leftX - 2*\loopR}, \pY) -- \mv;
 
       \draw[very thick]  \mv -- ({\leftX - 3*\loopR}, \pY);       

}

\nc{\makeBaseTwo}{

       \draw[very thick]   ({\leftX + \loopR}, \pY) -- \mv;
       
       \begin{scope}[very thick,decoration={
             markings,
             mark=at position 0.5 with {\arrow{>}}}
         ]  

         \draw[postaction={decorate}]  \mv -- (\leftX, \pY) node[pos=.5,above]{$l$};

       \end{scope}
  
       \draw[very thick]   ({\leftX - 2*\loopR}, \pY) -- \mv;
}

\nc{\bRight}{
        \begin{scope}[very thick,decoration={
           markings,
           mark=at position 0.5 with {\arrow{>}}}
       ]  
       \draw[postaction={decorate}]   plot[domain=360:180] ({\loopR*cos(\x) + \leftX + 3*\loopR}, {\loopR*sin(\x) + \pY});
       \node[below] at  ({\loopR*cos(270) + \leftX + 3*\loopR}, {\loopR*sin(270) + \pY}) {$g$};
     \end{scope}
   }

\nc{\makeBraid}{

     \makeBase

     \begin{scope}[very thick,decoration={
           markings,
           mark=at position 0.5 with {\arrow{>}}}
       ]  

       \draw[postaction={decorate}]   plot[domain=180:0]  
              ({\loopR*cos(\x) + \leftX}, {\loopR*sin(\x) + \pY});
       \node[above] at  ({\loopR*cos(90) + \leftX}, {\loopR*sin(90) + \pY}) {$k$};

       \draw[postaction={decorate}]     plot[domain=0:180] ({2*\loopR*cos(\x) + \leftX}, {2*\loopR*sin(\x) + \pY});
       \node[above] at  ({2*\loopR*cos(90) + \leftX}, {2*\loopR*sin(90) + \pY}) {$g$};
       \draw[postaction={decorate}]      plot[domain=180:0] ({4*\loopR*cos(\x) + \leftX}, {4*\loopR*sin(\x) + \pY});
       \node[above] at  ({4*\loopR*cos(90) + \leftX}, {4*\loopR*sin(90) + \pY}) {$g$};

       \draw[postaction={decorate}]      plot[domain=180:0] ({3*\loopR*cos(\x) + \leftX}, {3*\loopR*sin(\x) + \pY});
       \node[above] at  ({3*\loopR*cos(90) + \leftX}, {3*\loopR*sin(90) + \pY}) {$h$};

     \end{scope}

     \bRight
       
}

\nc{\makeBraidOne}{

     \makeBase

     \begin{scope}[very thick,decoration={
           markings,
           mark=at position 0.25 with {\arrow{>}},
          mark=at position 0.75 with {\arrow{>}}}
       ]  

       \draw[postaction={decorate}]   plot[domain=180:0]  
              ({\loopR*cos(\x) + \leftX}, {\loopR*sin(\x) + \pY});
       \node[above] at  ({\loopR*cos(45) + \leftX}, {\loopR*sin(45) + \pY}) {$k$};
       \node[above] at  ({\loopR*cos(135) + \leftX}, {\loopR*sin(135) + \pY}) {$k$};

       \draw[postaction={decorate}]     plot[domain=0:180] ({2*\loopR*cos(\x) + \leftX}, {2*\loopR*sin(\x) + \pY});
       \node[above] at  ({2*\loopR*cos(45) + \leftX}, {2*\loopR*sin(45) + \pY}) {$g$};
       \node[above] at  ({2*\loopR*cos(135) + \leftX}, {2*\loopR*sin(135) + \pY}) {$g$};
       \draw[postaction={decorate}]      plot[domain=180:0] ({4*\loopR*cos(\x) + \leftX}, {4*\loopR*sin(\x) + \pY});
       \node[above] at  ({4*\loopR*cos(45) + \leftX}, {4*\loopR*sin(45) + \pY}) {$g$};
       \node[above] at  ({4*\loopR*cos(135) + \leftX}, {4*\loopR*sin(135) + \pY}) {$g$};

       \draw[postaction={decorate}]      plot[domain=180:0] ({3*\loopR*cos(\x) + \leftX}, {3*\loopR*sin(\x) + \pY});
       \node[above] at  ({3*\loopR*cos(45) + \leftX}, {3*\loopR*sin(45) + \pY}) {$h$};
       \node[above] at  ({3*\loopR*cos(135) + \leftX}, {3*\loopR*sin(135) + \pY}) {$h$};

     \end{scope}

     \bRight       
}

\nc{\makeBraidTwo}{

     \makeBase

     \begin{scope}[very thick,decoration={
           markings,
           mark=at position 0.1 with {\arrow{>}},
          mark=at position 0.9 with {\arrow{>}}}
       ]  
       \draw[postaction={decorate}]      plot[domain=180:0] ({3*\loopR*cos(\x) + \leftX}, {3*\loopR*sin(\x) + \pY});
       \node[right] at  ({3*\loopR*cos(18) + \leftX}, {3*\loopR*sin(18) + \pY}) {$h$};
       \node[left] at  ({3*\loopR*cos(162) + \leftX}, {3*\loopR*sin(162) + \pY}) {$h$};

       \draw[postaction={decorate}]      plot[domain=180:0] ({\loopR*cos(\x) + \leftX}, {3*\loopR*sin(\x) + \pY});
       \node[left] at  ({\loopR*cos(18) + \leftX}, {3*\loopR*sin(18) + \pY}) {$k$};
       \node[left] at  ({\loopR*cos(162) + \leftX}, {3*\loopR*sin(162) + \pY}) {$k$};
     \end{scope}

     \begin{scope}[very thick,decoration={
           markings,
           mark=at position 0.9 with {\arrow{>}}}]
       \draw[postaction={decorate}]         plot[domain=0:180] ({2*\loopR*cos(\x) + \leftX}, {3*\loopR*sin(\x) + \pY});
       \node[left] at  ({2*\loopR*cos(162) + \leftX}, {3*\loopR*sin(162) + \pY}) {$g$};
     \end{scope}
     \begin{scope}[very thick,decoration={
           markings,
           mark=at position 0.1 with {\arrow{>}}}]
     \draw[postaction={decorate}]       plot[domain=180:0] ({4*\loopR*cos(\x) + \leftX}, {3*\loopR*sin(\x) + \pY});
       \node[left] at  ({4*\loopR*cos(162) + \leftX}, {3*\loopR*sin(162) + \pY}) {$g$};
     \end{scope}
         
       \bRight
}

\nc{\makeBraidThree}{

  \makeBaseTwo

     \begin{scope}[very thick,decoration={
           markings,
          mark=at position 0.8 with {\arrow{>}}}
       ]  
       \draw[postaction={decorate}]      plot[domain=90:0] ({3*\loopR*cos(\x) + \leftX}, {3*\loopR*sin(\x) + \pY});
       \node[right] at  ({3*\loopR*cos(18) + \leftX}, {3*\loopR*sin(18) + \pY}) {$h$};

       \draw[postaction={decorate}]      plot[domain=90:0] ({\loopR*cos(\x) + \leftX}, {3*\loopR*sin(\x) + \pY});
       \node[left] at  ({\loopR*cos(18) + \leftX}, {3*\loopR*sin(18) + \pY}) {$k$};
     \end{scope}

     \begin{scope}[very thick,decoration={
           markings,
           mark=at position 0.1 with {\arrow{>}}}]
       \draw[postaction={decorate}]         plot[domain=180:0] ({2*\loopR*cos(\x) + \leftX}, {3*\loopR*sin(\x) + \pY});
       \node[left] at  ({2*\loopR*cos(162) + \leftX}, {3*\loopR*sin(162) + \pY}) {$kg^{-1}hg$};
     \end{scope}

     \draw[very thick]       plot[domain=90:0] ({4*\loopR*cos(\x) + \leftX}, {3*\loopR*sin(\x) + \pY});
         
       \bRight
}

\begin{figure}
\centering

\begin{tabular}{|c|c|}

  \hline
  
    \begin{tikzpicture}[scale=2.5]

    \begin{scope}[very thick,decoration={
    markings,
    mark=at position 0.5 with {\arrow{>}}}
    ]

       \draw  \mv -- ({\leftX - \loopR}, \pY);
       \draw[postaction={decorate}]    plot[domain=180:0] ({\loopR*cos(\x) + \leftX}, {\loopR*sin(\x) + \pY});
       \node[above] at ({\loopR*cos(90) + \leftX}, {\loopR*sin(90) + \pY}) {$g$};
       \draw   ({\leftX + \loopR}, \pY) -- \mv;
       
       \draw[postaction={decorate}]  \mv -- (\leftX, \pY) node[pos=.5,above]{$h$};

       \draw  \mv -- ({\rightX - \loopR}, \pY);
       \draw[postaction={decorate}] plot[domain=180:0] ({\loopR*cos(\x) + \rightX}, {\loopR*sin(\x) + \pY});
       \node[above] at ({\loopR*cos(90) + \rightX}, {\loopR*sin(90) + \pY}) {$k$};
       \draw   ({\rightX + \loopR}, \pY) -- \mv;
 
       \draw[postaction={decorate}]  \mv -- (\rightX, \pY) node[pos=.5, above]{$l$};       
  
    \end{scope}
    \end{tikzpicture}

&

    \begin{tikzpicture}[scale=2.5]

    \begin{scope}[very thick,decoration={
    markings,
    mark=at position 0.5 with {\arrow{>}}}
    ]  

    \makeBraid
  
    \end{scope}
    \end{tikzpicture}

\\ \hline

    \begin{tikzpicture}[scale=2.5]

      \makeBraidOne
         
    \begin{scope}[very thick,decoration={
    markings,
    mark=at position 0.5 with {\arrow{>}}}
    ]  

       \draw ({\leftX}, {\pY + \loopR}) -- ({\leftX}, {\pY + 4*\loopR});
  
    \end{scope}
    \end{tikzpicture}

&

    \begin{tikzpicture}[scale=2.5]

      \makeBraidTwo     
         
    \end{tikzpicture}

\\ \hline

    \begin{tikzpicture}[scale=2.5]

      \makeBraidThree
         
    \end{tikzpicture}

&

    \begin{tikzpicture}[scale=2.5]

    \begin{scope}[very thick,decoration={
    markings,
    mark=at position 0.5 with {\arrow{>}}}
    ]

       \draw  \mv -- ({\leftX - \loopR}, \pY);
       \draw[postaction={decorate}]    plot[domain=180:0] ({\loopR*cos(\x) + \leftX}, {\loopR*sin(\x) + \pY});
       \node[above] at ({\loopR*cos(90) + \leftX}, {\loopR*sin(90) + \pY}) {$kg^{-1}hg$};
       \draw   ({\leftX + \loopR}, \pY) -- \mv;
       
       \draw[postaction={decorate}]  \mv -- (\leftX, \pY) node[pos=.5,above]{$l$};

       \draw  \mv -- ({\rightX - \loopR}, \pY);
       \draw[postaction={decorate}] plot[domain=180:0] ({\loopR*cos(\x) + \rightX}, {\loopR*sin(\x) + \pY});
       \node[above] at ({\loopR*cos(90) + \rightX}, {\loopR*sin(90) + \pY}) {$g$};
       \draw   ({\rightX + \loopR}, \pY) -- \mv;
 
       \draw[postaction={decorate}]  \mv -- (\rightX, \pY) node[pos=.5, above]{$h$};       
  
    \end{scope}
    \end{tikzpicture}

\\ \hline
\end{tabular}

\caption{Using local moves to calculate the action of a braid generator on an arbitrary element of the spanning set $S$. Read from left to right, then top to bottom.   Unlabeled interior edges are colored by the group identity element. The first two subfigures show application of the braid generator, which interchanges the univalent vertices. The last four show the local moves relating the image to another element of $S$.}
\label{fig:braid1}
\end{figure}
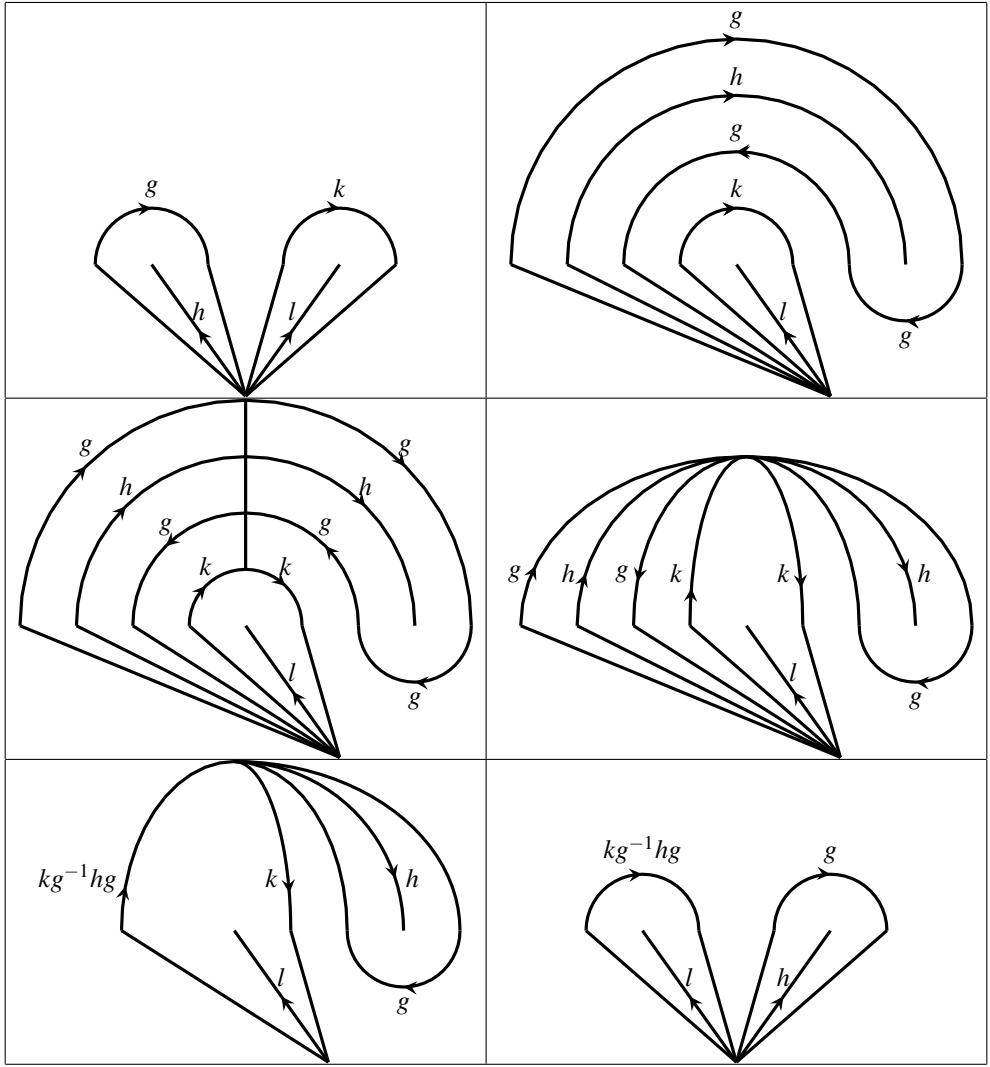


\rnc{\mv}{(0,0)}
\nc{\uv}{(0, 0.7)}
\nc{\rvx}{0.3}
\nc{\rvy}{\lcuty}
\nc{\rv}{(\rvx, \rvy)}
\nc{\lv}{(-\rvx, \lcuty)}

\nc{\cutRatio}{0.5}
\rnc{\lcutx}{{\cutRatio*(-0.5) + (1-\cutRatio)*(-1)}}
\rnc{\lcuty}{{\cutRatio*(0.866)}}
\rnc{\lcut}{(\lcutx, \lcuty)}
\rnc{\rcutx}{{\cutRatio*(0.5) + (1-\cutRatio)*(1)}}
\rnc{\rcut}{(\rcutx, \lcuty)}
\nc{\boundaryComponentx}{0}
\nc{\boundaryComponenty}{\lcuty}
\nc{\boundaryComponent}{(\boundaryComponentx, \boundaryComponenty)}

\nc{\connectorEnd}{({0.5*(-0.85)}, {0.5*0.2598 + 0.5*0.433})}

\nc{\edgeRatio}{0.9}
\nc{\lex}{{\edgeRatio*(-1) + (1-\edgeRatio)*(\lcx)}}
\nc{\ley}{{(1-\edgeRatio)*(\lcy)}}
\nc{\lev}{(\lex, \ley)}
\nc{\rex}{{\edgeRatio*(1) + (1-\edgeRatio)*(\rcx)}}
\nc{\rev}{(\rex, \ley)}

\nc{\eOneRatio}{0.8}
\nc{\leOnex}{{\eOneRatio*(-1) + (1-\eOneRatio)*(\lcx)}}
\nc{\leOney}{{(1-\eOneRatio)*(\lcy)}}
\nc{\leOnev}{(\leOnex, \leOney)}
\nc{\reOnex}{{\eOneRatio*(1) + (1-\eOneRatio)*(\rcx)}}
\nc{\reOnev}{(\reOnex, \leOney)}

\nc{\eTwoRatio}{0.7}
\nc{\leTwox}{{\eTwoRatio*(-1) + (1-\eTwoRatio)*(\lcx)}}
\nc{\leTwoy}{{(1-\eTwoRatio)*(\lcy)}}
\nc{\leTwov}{(\leTwox, \leTwoy)}
\nc{\reTwox}{{\eTwoRatio*(1) + (1-\eTwoRatio)*(\rcx)}}
\nc{\reTwov}{(\reTwox, \leTwoy)}

\nc{\eThreeRatio}{0.6}
\nc{\leThreex}{{\eThreeRatio*(-1) + (1-\eThreeRatio)*(\lcx)}}
\nc{\leThreey}{{(1-\eThreeRatio)*(\lcy)}}
\nc{\leThreev}{(\leThreex, \leThreey)}
\nc{\reThreex}{{\eThreeRatio*(1) + (1-\eThreeRatio)*(\rcx)}}
\nc{\reThreev}{(\reThreex, \leThreey)}

\nc{\eFourRatio}{0.4}
\nc{\leFourx}{{\eFourRatio*(-1) + (1-\eFourRatio)*(\lcx)}}
\nc{\leFoury}{{(1-\eFourRatio)*(\lcy)}}
\nc{\leFourv}{(\leFourx, \leFoury)}
\nc{\reFourx}{{\eFourRatio*(1) + (1-\eFourRatio)*(\rcx)}}
\nc{\reFourv}{(\reFourx, \leFoury)}

\nc{\eFiveRatio}{0.2}
\nc{\leFivex}{{\eFiveRatio*(-1) + (1-\eFiveRatio)*(\lcx)}}
\nc{\leFivey}{{(1-\eFiveRatio)*(\lcy)}}
\nc{\leFivev}{(\leFivex, \leFivey)}
\nc{\reFivex}{{\eFiveRatio*(1) + (1-\eFiveRatio)*(\rcx)}}
\nc{\reFivev}{(\reFivex, \leFivey)}

\nc{\dby}{-0.2} 

\begin{figure}
\centering

\begin{tabular}{|c|c|}

  \hline
  
    \begin{tikzpicture}[scale=3]    

      \makeBdy

    \begin{scope}[very thick,decoration={
    markings,
    mark=at position 0.5 with {\arrow{>}}}
    ]  

      \draw[loosely dashed] \lcut -- \rcut;

      \draw[postaction={decorate}]   (0, \dby) -- \mv node[pos=.5,right]{$p:=hglkl^{-1}h^{-1}$};
      \draw[postaction={decorate}]   \mv -- \rev node[pos=.5,sloped,above]{$h$};
      \draw[postaction={decorate}]   \mv -- \rv node[pos=.5,right]{$gl$};
      \draw[postaction={decorate}]   \mv -- \boundaryComponent node[pos=.5,left]{$k$};
      \draw[postaction={decorate}]   \lv -- \mv node[pos=.5,left]{$l$};
      \draw[postaction={decorate}]   \lev -- \mv node[pos=.5,sloped,below]{$h$};

      \draw[postaction={decorate}]   \rv -- \uv;
      \draw[postaction={decorate}]   \uv -- \mtop node[pos=.5,right]{$g$};
      \draw[postaction={decorate}]   \uv -- \lv;

    \end{scope}
    \end{tikzpicture}

&

    \begin{tikzpicture}[scale=3]    

      \makeBdy

    \begin{scope}[very thick,decoration={
    markings,
    mark=at position 0.5 with {\arrow{>}}}
    ]

      \draw[postaction={decorate}]   (0, \dby) -- \mv node[pos=.5,right]{$p$};
      \draw[postaction={decorate}]   \mv -- \rev node[pos=.5,sloped,below]{$h$};

      \draw[postaction={decorate}]   \mv -- \reOnev node[pos=.8,sloped,red]{$gl$};
      \draw[postaction={decorate}]   \leOnev -- \rv node[pos=.5,sloped,below]{$gl$};
      \draw[postaction={decorate}]   \rv -- \leFivev node[pos=.5,sloped,above]{$gl$};
      \draw[postaction={decorate}]   \reFivev -- \uv node[pos=.5,sloped,above]{$gl$};

      \draw[postaction={decorate}]   \mv -- \reTwov node[pos=.6,sloped,red]{$k$};
      \draw[postaction={decorate}]   \leTwov -- \boundaryComponent node[pos=.5,sloped,red]{$k$};

      \draw[postaction={decorate}]   \reThreev -- \mv node[pos=.5,sloped,above]{$l$};
      \draw[postaction={decorate}]   \lv -- \leThreev;
      \draw[postaction={decorate}]   \leFourv -- \lv node[pos=.5,sloped,above]{$l$};
      \draw[postaction={decorate}]   \uv -- \reFourv node[pos=.5,sloped,below]{$l$};

      \draw[postaction={decorate}]   \lev -- \mv node[pos=.5,sloped,below]{$h$};

      \draw[postaction={decorate}]   \uv -- \mtop node[pos=.5,right]{$g$};

    \end{scope}
    \end{tikzpicture}

\\ \hline

    \begin{tikzpicture}[scale=3]    

      \makeBdy

    \begin{scope}[very thick,decoration={
    markings,
    mark=at position 0.5 with {\arrow{>}}}
    ]

      \draw[postaction={decorate}]   (0, \dby) -- \mv node[pos=.5,right]{$p$};
      \draw[postaction={decorate}]   \mv -- \rev node[pos=.5,sloped,below]{$h$};

      \draw[postaction={decorate}]   \mv -- \reOnev node[pos=.8,sloped,red]{$gl$};
      \draw[postaction={decorate}]   \leOnev --  (-0.35, 0.3031) node[pos=.5,sloped,below]{$gl$};
      \draw[postaction={decorate}]   (-0.35, 0.3031) -- \rv node[pos=.5,sloped,below]{$gl$};

      \draw[postaction={decorate}]   \rv -- \leFivev node[pos=.5,sloped,above]{$gl$};
      \draw[postaction={decorate}]   \reFivev -- \uv node[pos=.5,sloped,above]{$gl$};

      \draw[postaction={decorate}]   \mv -- \reTwov node[pos=.6,sloped,red]{$k$};
      \draw[postaction={decorate}]   \leTwov -- \connectorEnd node[pos=.5,sloped,red]{$k$};
      \draw[postaction={decorate}]   \connectorEnd -- \boundaryComponent node[pos=.5,sloped,red]{$k$};

      \draw[postaction={decorate}]   \reThreev -- \mv node[pos=.5,sloped,above]{$l$};
      \draw[postaction={decorate}]   \lv -- \leThreev;
      \draw[postaction={decorate}]   \leFourv -- \lv node[pos=.5,sloped,above]{$l$};
      \draw[postaction={decorate}]   \uv -- \reFourv node[pos=.5,sloped,below]{$l$};

      \draw[postaction={decorate}]   \lev -- \mv node[pos=.5,sloped,below]{$h$};

      \draw[postaction={decorate}]   \uv -- \mtop node[pos=.5,right]{$g$};

      \draw \lcut -- \lv;
      \draw \rv -- \rcut;
      \draw \mv -- \connectorEnd;
     
    \end{scope}
    \end{tikzpicture}

&

    \begin{tikzpicture}[scale=3]    

      \makeBdy

    \begin{scope}[very thick,decoration={
    markings,
    mark=at position 0.5 with {\arrow{>}}}
    ]

      \draw[postaction={decorate}]   (0, \dby) -- \mv node[pos=.5,right]{$p$};
      \draw[postaction={decorate}]   \mv -- \rev node[pos=.5,sloped,below]{$h$};

      \draw[postaction={decorate}]   \mv -- \reOnev node[pos=.5,sloped,above]{$glk$};
      \draw[postaction={decorate}]   \leOnev --  (-0.35, 0.3031) node[pos=.5,sloped,above]{$glk$};
      \draw[postaction={decorate}]   (-0.35, 0.3031) -- \rv node[pos=.5,sloped,below]{$gl$};

      \draw[postaction={decorate}]   \rv -- \leFivev node[pos=.5,sloped,above]{$gl$};
      \draw[postaction={decorate}]   \reFivev -- \uv node[pos=.5,sloped,above]{$gl$};

      \draw[postaction={decorate}]    (-0.35, 0.3031)  -- \boundaryComponent node[pos=.5,sloped,above]{$k$};

      \draw[postaction={decorate}]   \rv -- \mv node[pos=.5,right]{$l$};
      \draw[postaction={decorate}]   \uv -- \rv node[pos=.5,right]{$l$};

      \draw[postaction={decorate}]   \lev -- \mv node[pos=.5,sloped,below]{$h$};

      \draw[postaction={decorate}]   \uv -- \mtop node[pos=.5,right]{$g$};

      \draw \mv --  (-0.35, 0.3031);
     
    \end{scope}
    \end{tikzpicture}

\\ \hline


 \begin{tikzpicture}[scale=3]    

      \makeBdy

    \begin{scope}[very thick,decoration={
    markings,
    mark=at position 0.5 with {\arrow{>}}}
    ]

      \draw[postaction={decorate}]   (0, \dby) -- \mv node[pos=.5,right]{$p$};
      \draw[postaction={decorate}]   \mv -- \rev node[pos=.5,sloped,below]{$hglk$};

      \draw[postaction={decorate}]   \rv -- \leFivev node[pos=.5,sloped,above]{$gl$};
      \draw[postaction={decorate}]   \reFivev -- \uv node[pos=.5,sloped,above]{$gl$};

      \draw[postaction={decorate}]    \mv  -- \boundaryComponent node[pos=.5,left]{$k$};

      \draw[postaction={decorate}]   \mv -- \rv node[pos=.5,right]{$l^{-1}gl$};
      \draw[postaction={decorate}]   \uv -- \rv node[pos=.5,right]{$l$};

      \draw[postaction={decorate}]   \lev -- \mv node[pos=.5,sloped,below]{$hglk$};

      \draw[postaction={decorate}]   \uv -- \mtop node[pos=.5,right]{$g$};
     
    \end{scope}
    \end{tikzpicture}

&


 \begin{tikzpicture}[scale=3]    

      \makeBdy

    \begin{scope}[very thick,decoration={
    markings,
    mark=at position 0.5 with {\arrow{>}}}
    ]

      \draw[postaction={decorate}]   (0, \dby) -- \mv node[pos=.5,right]{$p$};
      \draw[postaction={decorate}]   \mv -- \rev node[pos=.5,sloped,below]{$hglk$};

      \draw[postaction={decorate}]   \uv -- \leFivev node[pos=.5,sloped,below]{$gl$};
      \draw[postaction={decorate}]   \reFivev -- \uv node[pos=.5,sloped,below]{$gl$};

      \draw[postaction={decorate}]    \mv  -- \boundaryComponent node[pos=.5,left]{$k$};

      \draw[postaction={decorate}]   \mv -- \rv node[pos=.5,right]{$l^{-1}gl$};
      \draw[postaction={decorate}]   \rv -- \uv;

      \draw[postaction={decorate}]   \lev -- \mv node[pos=.5,sloped,below]{$hglk$};

      \draw[postaction={decorate}]   \uv -- \mtop node[pos=.5,right]{$g$};

    \end{scope}
    \end{tikzpicture}

\\ \hline


   \begin{tikzpicture}[scale=3]    

      \makeBdy

    \begin{scope}[very thick,decoration={
    markings,
    mark=at position 0.5 with {\arrow{>}}}
    ]

      \draw[postaction={decorate}]   (0, \dby) -- \mv node[pos=.5,right]{$p$};
      \draw[postaction={decorate}]   \mv -- \rev node[pos=.5,sloped,below]{$hglk$};

      \draw[postaction={decorate}]   \rev -- \uv node[pos=.5,sloped,above]{$gl$};
      \draw[postaction={decorate}]   \uv -- \lev node[pos=.5,sloped,above]{$gl$};

      \draw[postaction={decorate}]    \mv  -- \boundaryComponent node[pos=.5,left]{$k$};

      \draw[postaction={decorate}]   \mv -- \rv node[pos=.5,right]{$l^{-1}gl$};
      \draw[postaction={decorate}]   \rv -- \uv;

      \draw[postaction={decorate}]   \lev -- \mv node[pos=.5,sloped,below]{$hglk$};
      
      \draw[postaction={decorate}]   \uv -- \mtop node[pos=.5,right]{$g$};;

    \end{scope}
    \end{tikzpicture}

&

 \begin{tikzpicture}[scale=3]    

      \makeBdy

    \begin{scope}[very thick,decoration={
    markings,
    mark=at position 0.5 with {\arrow{>}}}
    ]

      \draw[postaction={decorate}]   (0, \dby) -- \mv node[pos=.5,right]{$p$};
      \draw[postaction={decorate}]   \mv -- \rev node[pos=.5,sloped,below]{$hglk$};

      \draw[postaction={decorate}]   \uv -- \leTwov  node[pos=.5,sloped,below]{$gl$};
      \draw[postaction={decorate}]   \reTwov -- \mv  node[pos=.5,sloped,above]{$gl$};

      \draw[postaction={decorate}]    \mv  -- \boundaryComponent node[pos=.5,left]{$k$};

      \draw[postaction={decorate}]   \mv -- \rv;
      \draw[postaction={decorate}]   \rv -- \uv node[pos=.5,right]{$g^2l$};

      \draw[postaction={decorate}]   \lev -- \mv node[pos=.5,sloped,below]{$hglk$};
      
      \draw[postaction={decorate}]   \uv -- \mtop node[pos=.5,right]{$g$};

    \end{scope}
    \end{tikzpicture}

\\ \hline

 \begin{tikzpicture}[scale=3]    

      \makeBdy

    \begin{scope}[very thick,decoration={
    markings,
    mark=at position 0.5 with {\arrow{>}}}
    ]

      \draw[postaction={decorate}]   (0, \dby) -- \mv node[pos=.5,right]{$p$};
      \draw[postaction={decorate}]   \mv -- \rev node[pos=.5,sloped,below]{$hglk$};

      \draw[postaction={decorate}]   \uv -- \leTwov node[pos=.5,sloped,below]{$gl$};
      \draw[postaction={decorate}]   \reTwov -- \mv node[pos=.5,sloped,above]{$gl$};
      \draw \leTwov -- \mv;

      \draw[postaction={decorate}]    \mv  -- \boundaryComponent node[pos=.5,left]{$k$};

      \draw[postaction={decorate}]   \mv -- \rv;
      \draw[postaction={decorate}]   \rv -- \uv node[pos=.5,right]{$g^2l$};

      \draw[postaction={decorate}]   \lev -- \mv node[pos=.5,sloped,below]{$hglk$};
      
      \draw[postaction={decorate}]   \uv -- \mtop node[pos=.5,right]{$g$};

    \end{scope}
    \end{tikzpicture}

&


 \begin{tikzpicture}[scale=3]    

      \makeBdy

    \begin{scope}[very thick,decoration={
    markings,
    mark=at position 0.5 with {\arrow{>}}}
    ]

      \draw[postaction={decorate}]   (0, \dby) -- \mv node[pos=.5,right]{$p$};
      \draw[postaction={decorate}]   \mv -- \rev node[pos=.5,sloped,below]{$hglkl^{-1}g^{-1}$};

      \draw[postaction={decorate}]   \uv -- \lv;
      \draw[postaction={decorate}]   \lv -- \mv node[pos=.5,left]{$gl$};

      \draw[postaction={decorate}]    \mv  -- \boundaryComponent node[pos=.5,left]{$k$};

      \draw[postaction={decorate}]   \mv -- \rv;
      \draw[postaction={decorate}]   \rv -- \uv node[pos=.5,right]{$g^2l$};

      \draw[postaction={decorate}]   \lev -- \mv node[pos=.5,sloped,below]{$hglkl^{-1}g^{-1}$};
  
      \draw[postaction={decorate}]   \uv -- \mtop node[pos=.5,right]{$g$};

    \end{scope}
    \end{tikzpicture}

\\ \hline
\end{tabular}

\caption{Using local moves to calculate the action of the last Birman generator on an arbitrary element of the spanning set $S$.   This generator corresponds to pulling a boundary component of the surface $\Si$ along a generator for the fundamental group of the closed surface given by filling in all boundary components of $\Si$. Read from left to right, then top to bottom.   Unlabeled interior edges are colored by the group identity element. The first two figures show application of the Birman generator. The last eight show the local moves relating the image to another element of $S$.}

\label{fig:drag1}
\end{figure}
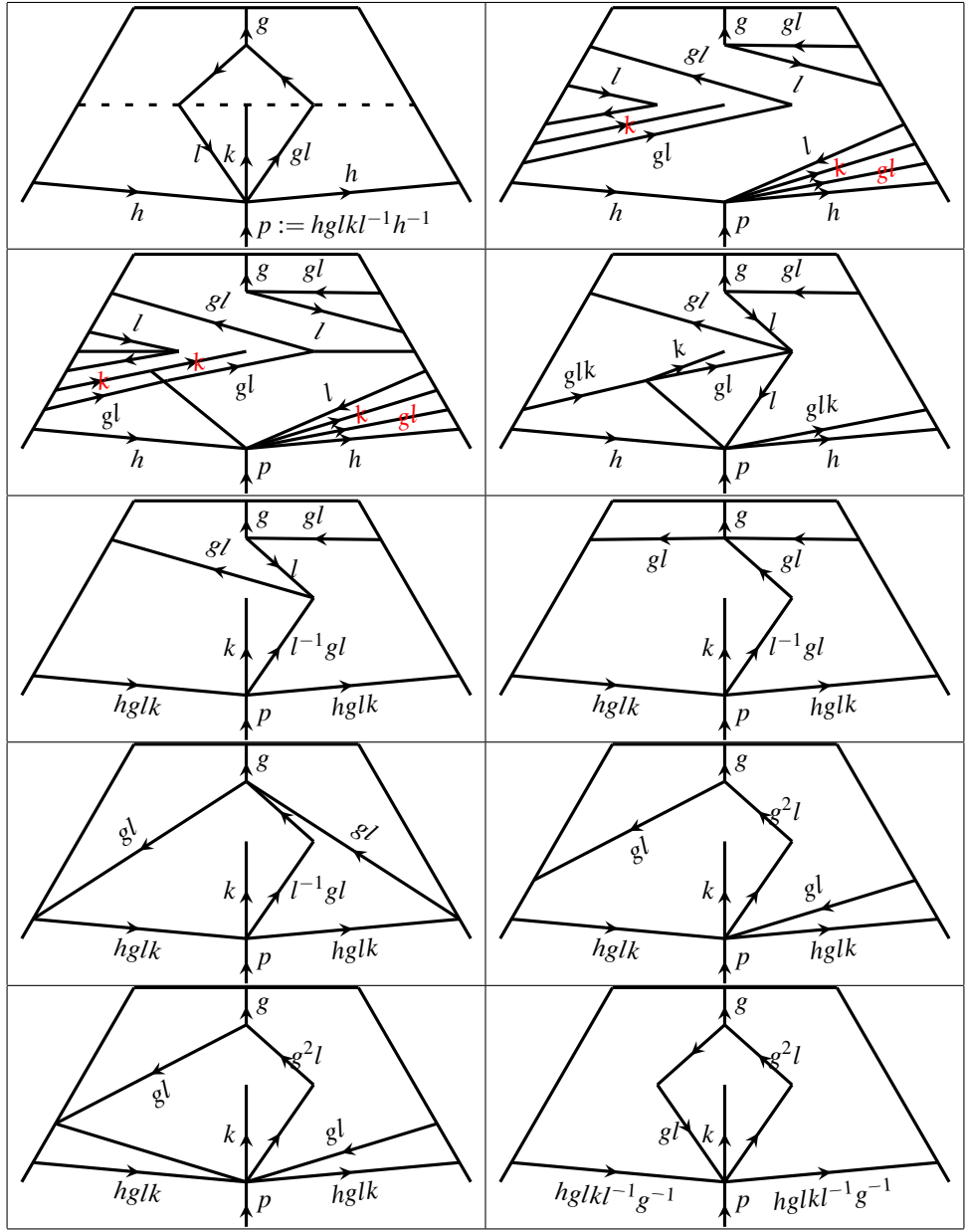

\begin{thm}\label{thm:compact}
The image of any twisted Dijkgraaf-Witten representation of a mapping class group of an orientable, compact surface with boundary is finite.  In particular, the image of any such braid group representation is finite.
\end{thm}

\section{Further Directions}
We have proved that every twisted Dijkgraaf-Witten representation of a mapping class group of a compact, orientable surface has finite image.  This is a generalization of the results of \cite{erw} and \cite{fjfu}, as well as another step towards the (modified) Property F conjecture. A potential next step would be to consider more complicated spherical categories than $\vgo$.  One candidate is the class of Tambara-Yamagami categories \cite{tambara}.  The main additional complication here is the appearance of multifusion channels, i.e. the tensor product of two simple objects can be a direct sum of multiple simple objects.  

\medskip


\end{document}